\documentclass[english,reqno]{amsart}

\usepackage{amsmath}
\usepackage{amssymb} 
\usepackage{enumitem} 
\usepackage{mathtools} 
\usepackage[table]{xcolor} 
\usepackage[all]{xy} 
\usepackage{tikz} 
\usepackage{tikz-cd}
\usepackage{indentfirst} 
\usepackage{babel} 
\usepackage{setspace} 

\usepackage[colorlinks,linkcolor=red,anchorcolor=green,citecolor=blue]{hyperref} 
\hypersetup{linktocpage = true} 

\usepackage{rotating} 

\usepackage{ytableau} 
\usepackage{longtable} 
\newcolumntype{M}[1]{>{\centering\arraybackslash}m{#1}} 

\usepackage{mathpazo}
\usepackage{mathrsfs}
\DeclareFontFamily{OMS}{rsfs}{\skewchar\font'60}
\DeclareFontShape{OMS}{rsfs}{m}{n}{<-5>rsfs5 <5-7>rsfs7 <7->rsfs10 }{}
\DeclareSymbolFont{rsfs}{OMS}{rsfs}{m}{n}
\DeclareSymbolFontAlphabet{\scr}{rsfs}
\DeclareSymbolFontAlphabet{\scr}{rsfs}

\usepackage[T1]{fontenc}

\newcommand\cE{{\mathcal E}}
\newcommand\cF{{\mathcal F}}
\newcommand\cG{{\mathcal G}}
\newcommand\cH{{\mathcal H}}

\newcommand\cO{{\mathcal O}}

\newcommand\cX{{\mathcal X}}

\newcommand\dbar{{\overline{\partial}}}

\theoremstyle{plain}
\newtheorem{thm}{Theorem}[section]
\newtheorem{lemma}[thm]{Lemma}
\newtheorem{prop}[thm]{Proposition}
\newtheorem{cor}[thm]{Corollary}

\newtheorem{assumption}[thm]{Assumption}
\newtheorem{conj}[thm]{Conjecture}

\theoremstyle{remark}

\newtheorem{remark}[thm]{Remark}

\setlist[itemize]{leftmargin=*}
\setlist[enumerate]{leftmargin=*}

\numberwithin{equation}{section} 

\makeatletter

\ifnum\@ptsize=0  \fi \ifnum\@ptsize=2  \fi

\title{ Admissible  metrics on compact  K\"ahler varieties} 

\subjclass[2020]{}

\author{Wenhao Ou}

\address{Wenhao Ou, Institute of Mathematics, Academy of Mathematics and Systems Science, Chinese Academy of Sciences, Beijing, 100190, China}
\email{wenhaoou@amss.ac.cn}

\begin{document}

\begin{abstract}
Let $X$ be a normal compact K\"ahler variety, and $\cF$ a coherent reflexive sheaf on $X$. 
We investigate the existence of admissible Hermitian metrics on $\cF$.  
If moreover $\cF$ is slope stable,  we also study the existence of admissible Hermitian-Yang-Mills metrics on it.  
The  existence   will hold if one can prove a uniform Sobolev inequality on singular spaces. 
\end{abstract}

\maketitle
 
\tableofcontents


\section{Introduction}

The theory of  holomorphic vector bundles is a central object in algebraic geometry and analytic geometry. 
The notion of stable vector bundles on complete curves was introduced by Mumford  in  \cite{Mumford1963}. 
Such notion of stability was then extended to torsion-free sheaves on any projective manifolds (see \cite{Takemoto1972}, \cite{Gieseker1977}), and is now known as slope stability, or Gieseker stability. 
An  important property of stable bundles is the following Bogomolov-Gieseker inequality, involving the Chern classes of the bundle.
\begin{thm}
\label{thm:BG-inequality-intro}
Let $X$ be a projective manifold of dimension $n$,  and let  $H$ be  an ample divisor.
Let $\cF$ be a $H$-stable vector bundle of rank $r$ on $X$. 
Then 
\[   \Big(c_2(\cF)-\frac{r-1}{2r}c_1(\cF)^2 \Big)  \cdot  H^{n-2} \geqslant 0. \]
\end{thm}
When  $X$ is a surface, the inequality was proved in \cite{Bogomolov1978}. 
In higher dimension, one may apply Mehta-Ramanathan   theorem in  \cite{MehtaRamanathan1981/82}   to  reduce to the case of surfaces, by taking hyperplane  sections. 
Later in \cite{Kawamata1992}, as a part of the proof for the three-dimensional abundance theorem, Kawamata extended  the inequality to reflexive sheaves on projective surfaces with quotient singularities. 
The same technique then  allows us to deduce the Bogomolov-Gieseker inequality for reflexive sheaves on projective varieties which has quotient singularities in codimension two.  

On the analytic side, let $(X,\omega)$ be a compact  K\"ahler manifold, and $(\cF, H)$ a Hermitian  holomorphic vector bundle on $X$.  
L\"ubke proved that  if $H$ satisfies the Einstein condition, then the following inequality holds (see \cite{Lubke1982}),
\[   \int_X \Big(c_2(\cF,H)-\frac{r-1}{2r}c_1(\cF,H)^2 \Big)  \wedge  \omega^{n-2} \geqslant 0. \]
We recall that $H$ satisfies the Einstein condition if its Chern curvature $F$ satisfies 
\[ \sqrt{-1}\Lambda F = \lambda \mathrm{Id}_\cF,\]
where $\Lambda$ is the contraction with $\omega$ and  $\lambda$ is a constant. 
In this case, $H$ is called a Hermitian-Einstein metric, or a Hermitian-Yang-Mills metric. 
It is now well understood that if $\cF$ is $\omega$-stable, then it admits a Hermitian-Yang-Mills metric. 
The case when $X$ is a complete curve was proved by Narasimhan-Seshadri  in \cite{NarasimhanSeshadri1965}, the case of projective surfaces was proved by Donaldson in \cite{Donaldson1985}, and the case of arbitrary compact K\"ahler manifolds was proved by Uhlenbeck-Yau in \cite{UhlenbeckYau1986}. 
As a consequence, one deduces a metric version for the proof of the previous Bogomolov-Gieseker inequality. 
An advantage of the analytic method is to give an insight to the equality condition. 
Simpson extended the existence of Hermitian-Yang-Mills metric to stable Higgs bundles, on compact and certain non compact K\"ahler manifolds, in \cite{Simpson1988}. 
Furthermore,  in \cite{BandoSiu1994},  Bando-Siu introduced the notion of admissible metrics and prove the existence of admissible Hermitian-Yang-Mills metrics on stable reflexive sheaves. 


In \cite{Donaldson1985}, Donaldson considered the following heat equation in metrics $H_t$ on a vector  bundle $\cF$. 
\begin{equation} \label{equation:metric-intro}
H_t^{-1}(\frac{\mathrm{\partial}}{\mathrm{\partial}t}H_t)  = -(\sqrt{-1}\Lambda  F_t - \lambda (\cF) \mathrm{Id}_\cF) 
\end{equation}
where   $\mu(\cF)$ is the slope of $\cF$,  $\lambda (\cF) =     \mu (\cF)\mathrm{Vol} (X) ^{-1}$  and  $F_t$ is the Chern curvature   of $H_t$. 
He   proved   this equation admits a solution until infinite time, when $X$ is a smooth projective surface. 
Then by using Yang-Mills theory, it was shown that, up to gauge transforms, there is a sequence 
$\{F_{t_i}\}$ which converges to a Hermitian-Yang-Mills connection $F$.  
Thanks to the theorem of removable singularities, the connection $F$ defines a holomorphic structure on the entire smooth vector bundle $\cF$. 
The stability condition then imposes that the new holomorphic structure coincides with the initial one, and hence $F$ is a Hermitian-Yang-Mills connection  for the holomorphic bundle $\cF$. 

The method of Uhlenbeck-Yau in \cite{UhlenbeckYau1986} is different. 
For a fixed  initial metric $H_0$ on the bundle $\cF$, they considered the following perturbed equations in metrics $H=H_0h$, 
\[   \sqrt{-1}\Lambda F_H -\lambda(\cF)\mathrm{Id}_\cF  - \epsilon \log\, h = 0.\] 
Such an  equation always admits a solution $h_\epsilon$ for $0<\epsilon \ll 1$. 
Then they argued by contradiction that, if $h_\epsilon$ are not uniformly $L^1$-bounded, then
there is a $L_1^2$-subbundle (or weakly holomorphic subbundle) of $\cF$, which induces a destabilizer for $\cF$. 
Once the uniform boundedness is obtained, one can extract a sequence $\{H_0h_{\epsilon_i}\}$ which converges to a Hermitian-Yang-Mills metric.

The method of Simpson in \cite{Simpson1988} is a  combination of both methods above. 
He  proved the heat solution (\ref{equation:metric-intro}) admits  a solution $H_t$  until infinite time. 
Then  arguing as in \cite{UhlenbeckYau1986}, he  deduced a sequence $\{H_{t_i}\}$ which converges to a Hermitian-Yang-Mills metric. 


The main objective of this note is to study  this problem for singular compact K\"ahler varieties $X$, with the methods of Bando-Siu and of Simpson. 
The idea is as follows.  
We take a desingularization $r\colon \widehat{X}\to X$ 
such that   $\widehat{\cF} = (r^*(\cF^*))^*$ becomes  locally free. 
Then we are able to solve the heat equations in $\widehat{X}$   for $\widehat{\cF}$, with respect to small perturbations $\omega_\epsilon$  of $r^*\omega$. 
When $X$ is smooth,  Bando-Siu   showed that these  solutions  $H_\epsilon$ have a limit point $H$ when $\epsilon$ tends to zero, and such $H$ is a solution for the initial  heat equation on $(X,\omega)$. 
This relies on the next two ingredients. 
First, the solutions $H_\epsilon$ admit  estimates  involving the heat kernels $K_\epsilon$ of $(X,\omega_\epsilon)$. 
Furthermore, Bando-Siu proved a uniform Sobolev inequality for  $(X,\omega_\epsilon)$, which implies a uniform estimate on the heat kernels $K_\epsilon$.

There would be some difficulties when one tries to adapt this method on singular varieties. 
Actually, the arguments involving local elliptic/parabolic estimates always remain valid. 
But one needs to be careful while getting global estimates. 
The first problem is the uniform Sobolev inequalities for singular varieties. 
We conjecture that it holds.

\begin{conj}\label{conj:Sobolev-conj-global}
Let $(X,\omega)$ be a compact K\"ahler variety of dimension $n$. 
Let $p\colon Y \to X$ be a blowup at smooth center, and $\eta$ a K\"ahler form on $Y$.    
For each $0< \epsilon \ll 1$, we  consider the  following K\"ahler form  on $Y$.  
\[\omega_{\epsilon } = r^*\omega+\epsilon  \eta. \] 
Then $(Y,\omega_\epsilon )$ has a uniform Sobolev constant for all $0<\epsilon  \leqslant 1$. 
That is, there exists a constant $C_S$, independent of $\epsilon $,    such that for any $\mathcal{C}^1$ function $f \geqslant 0$, compactly supported in the smooth locus of $Y$,   the following inequality holds,
\[ (\int_{Y} |f|^{\frac{2n}{ n-1} }\omega_\epsilon ^n)^{\frac{ n-1}{ n}}  \leqslant C_S \int_{Y} (||\nabla_{\omega_\epsilon } f||^2 + |f|^2) \omega_\epsilon ^n. \]
\end{conj}

Another difficulty is that, when $X$ has singularities, the heat kernel and the maximum principal are less trivial. 
To deal with this issue, we will investigate  heat kernels on singular compact K\"ahler varieties in Section \ref{section:heat-kernel}.  
Particularly, we prove the following assertion (see Corollary \ref{cor:heat-kernel-conservation}), extending a result of Li-Tian in \cite{LiTian1995}. 
It also implies \cite[Conjecture 3.1]{DiNezzaGuedjGuenancia2020} (see Corollary \ref{cor:DNGG}). 
\begin{thm}
\label{thm:heat-kernel-conservation}
Let $X^\circ \subseteq X$ be a  smooth Zariski open dense subset of  a compact K\"ahler variety  $(X,\omega)$, and $K $ the heat kernel of $(X^\circ, \omega)$.  
Then the following properties hold.
\begin{enumerate}
\item $K$ is bounded on $X^\circ \times X^\circ \times [t_1,t_2]$ for any fixed $0<t_1<t_2$. 
\item $K (x,\cdot,t)$ belongs to $L_1^2$.
\item \[\int_{y\in X^\circ}|\Delta_y K (x,y, t)|^2 \mathrm{d}y < \infty.\]
\item \[ 
\int_{y\in X^\circ}K (x,y,t) \mathrm{d}y = 1.
\] 
\end{enumerate}
\end{thm} 
\noindent Furthermore, instead of getting estimates of the limit solution $H$ directly from $(X,\omega)$, we will converge the estimates of $H_\epsilon$. 
For this reason, additional uniform estimates on $H_\epsilon$ are needed in our arguments.

After we obtain the solution of the heat equation, Simpson's method does not apply immediately yet. 
Three assumptions were made for the non compact manifolds in \cite{Simpson1988}, and it was proved that any Zariski open dense subset of a  compact K\"ahler manifold  satisfies the assumptions. 
It is not trivial  to us that this still holds while singularities are present. 
We will handle this problem in Section \ref{section:HYM-metric}. 

After all, if Conjecture \ref{conj:Sobolev-conj-global} holds, we can deduce the existence of admissible metrics on normal compact K\"ahler varieties.

\begin{thm}
\label{thm:existence-admissible}
Assume Conjecture \ref{conj:Sobolev-conj-global}. 
Let  $(X,\omega)$ be a compact normal K\"ahler variety. Let $\cF$ be a reflexive   sheaf on $X$. 
Then there is some initial Hermitian metric $H_0$ on $\cF$, such that  the heat equation (\ref{equation:metric-intro}) admits a solution until infinite time. 
Moreover, $H_t$ are admissible  for all $t>0$.
\end{thm}

\begin{thm}
\label{thm:existence-HE}
Assume Conjecture \ref{conj:Sobolev-conj-global}.
Let  $(X,\omega)$ be a normal compact  K\"ahler variety.
Let $\cF$ be a reflexive stable sheaf on $X$. 
Then there exists an admissible Hermitian-Yang-Mills metric on $\cF$.
\end{thm}
  
  
\begin{thm}
\label{thm:BG-inequality}
Assume Conjecture \ref{conj:Sobolev-conj-global}.
Let  $(X,\omega)$ be a normal compact  K\"ahler variety of dimension $n$.
Let $\cF$ be a coherent reflexive stable sheaf of rank $r$ on $X$. 
It admits an  admissible Hermitian-Yang-Mills metric $H$, 
and we have the following inequality
\[   \int_X \Big(c_2(\cF,H)-\frac{r-1}{2r}c_1(\cF,H)^2 \Big)  \wedge  \omega^{n-2} \geqslant 0. \]
The equality holds if and only if the Chern connection is projectively flat.
\end{thm}


This note  is organized as follows. We will first recall some basic notion in Section \ref{section:pre}. 
In Section \ref{section:initial-metric}, we will construct an initial metric for the heat equation. Such a metric also enable us to define the slope of a reflexive sheaf. 
In Section \ref{section:heat-kernel}, we will study heat kernels on compact K\"ahler varieties. 
We will prove Theorem \ref{thm:existence-admissible} in Section \ref{section:admissible}.
Finally in Section \ref{section:HYM-metric}, we will complete the proofs of  Theorem \ref{thm:existence-HE} and Theorem \ref{thm:BG-inequality}. \\

\noindent\textbf{Acknowledgment.} The author is indebted to Junyan Cao and Mihai P{\u a}un for enormous help. 
He is grateful to Chenjie Fan for sharing his analysis knowledge.   
He would thank Mitchell Faulk for sending him his PhD thesis.  
He would  thank Henri Guenancia and Hans-Joachim Hein for pointing out Corollary \ref{cor:DNGG} to him. 
He would also thank Omprokash Das, Andreas H\"oring and Burt Totaro for discussions and remarks.

\section{Preliminary}
\label{section:pre}

In this section, we will fix some conventions for the paper, and recall some elementary results. 
Throughout this paper,  we   use the symbol $\mathcal{C}^p$ (respectively $\mathcal{C}^{p,\alpha}$) for functions whose $p$-th derivatives exist and are continuous (respectively $\alpha$-H\"older continuous). 
We use the symbol $L^p$ for  functions whose $p$-th power is integrable,  $L^\infty$ for    bounded functions, and 
$L_i^p$ for    functions whose  derivatives, up to order $i$, are $L^p$.   
We denote the $L^p$-norm by $||\cdot ||_p$  for any  $0<p\leqslant \infty$.

\subsection{Possibly singular varieties.} 

For a  complex space $X$, we denote its smooth locus by $X_{reg}$.  
After \cite{Grauert1962}, a K\"ahler variety $(X,\omega)$ is an integral complex space $X$, together with a real closed $(1,1)$-form $\omega$,    
such that for any point $x\in X$, there is an open analytic  neighbourhood $U$ of  $x$ such that $U$ is isomorphic to a analytic subvariety of some open domain $V$ of $\mathbb{C}^N$, and $\omega|_U$ is the restriction of some K\"ahler form  on $V$. 
By a $L_i^p$ function $f$ on $X$, we mean  a function $f$ belonging to $L_i^p(X_{reg})$.

When $X$ is smooth, the we can define the gradient operator  $\nabla_\omega$ and the Laplace-Beltrami operator $\Delta_\omega$, with respect to the Riemannian metric induced by  $\omega$ on $X$. 
More precisely, we use the convention
\[\Delta_\omega = 2\sqrt{-1}\Lambda_\omega \partial \dbar,\] 
where $\Lambda_\omega$ is the contraction with $\omega$. 
In particular,  the Laplace-Beltrami operator has its sign as the one of $\frac{\partial^2}{(\partial x)^2}$.

The following lemma was established in the proof of \cite[Theorem 4.1 ]{LiTian1995}. 
We will use it to replace \cite[Assumption 2]{Simpson1988} in our demonstration on the existence of Hermitian-Yang-Mills metric.

\begin{lemma}\label{lemma:cut-off}
Let $M$ be a compact real variety without boundary, and $M^\circ \subseteq M$ a smooth   open subset. 
Assume that the real codimension of $M\backslash M^\circ$ is at least two.
Then there exists a  sequence of smooth cut-off functions $\{\varphi_i\}$ on $M^\circ$, with limit equal to  the constant function $1$,  such that $ ||\nabla {\varphi_i}||_{2} \to 0$  when $i\to \infty$.
\end{lemma}

As a consequence, we have the following two  formulas of integration by parts.

\begin{cor}\label{cor:integration-by-part}
Let $M$ be a compact real variety without boundary, and $M^\circ \subseteq M$ a smooth   open subset. 
Assume that the real codimension of $M\backslash M^\circ$ is at least two.
Then for  bounded functions $u,v \in L^2_1(M^\circ)$ such that  $\Delta v$ exists and is integrable, we have 
\[ 
\int_M \nabla u \cdot \nabla v  = - \int_M u\Delta v.
 \]
\end{cor}

\begin{proof}
We choose a  sequence of smooth cut-off functions $\{\varphi_i\}$ as in Lemma \ref{lemma:cut-off}. 
Then we have 
\[ 
\int_M \varphi_i \nabla u \cdot \nabla v  +  \int_M u \nabla  \varphi_i \cdot \nabla v  = - \int_M \varphi_i u\Delta v.
 \]
Since $u,v \in L^2_1(M)$ are assumed to be bounded, and since $\Delta v$ is integrable, we can take the limit for $i\to \infty$ in the previous equality, 
and deduce the corollary.
\end{proof}

\begin{cor}\label{cor:integration-by-part-2}
Let $M$ be a compact real variety without boundary, and $M^\circ \subseteq M$ a smooth   open subset. 
Assume that the real codimension of $M\backslash M^\circ$ is at least two. 
Then for a  smooth 1-form $\xi$ on $M^\circ$ with $||\xi||_\omega^2$ and $\mathrm{d}\xi$  integrable,  we have 
\[ 
\int_M \mathrm{d} \xi \wedge \omega^{n-1}  = 0.
 \]
\end{cor}

\begin{proof} 
We choose a  sequence of smooth cut-off functions $\{\varphi_i\}$ as in Lemma \ref{lemma:cut-off}. 
Then we have 
\[ 
0 = \int_M   \mathrm{d} (\varphi_i \xi) \wedge \omega^{n-1} = \int_M \mathrm{d}\varphi_i  \wedge \xi \wedge  \omega^{n-1} + \int_M \varphi_i \mathrm{d}  \xi \wedge  \omega^{n-1}
 \]
Since $||\xi||_\omega^2$ is integrable and since  $ ||\nabla {\varphi_i}||_{2} \to 0$  when $i\to \infty$, we see that 
\[\int_M  \mathrm{d}\varphi_i  \wedge \xi \wedge  \omega^{n-1} \to 0.\] 
Moreover, since $\mathrm{d}\xi$ is integrable, we have 
\[\int_M \varphi_i \mathrm{d}  \xi \wedge  \omega^{n-1} \to \int_M  \mathrm{d}  \xi \wedge  \omega^{n-1}.
\] 
The corollary then follows.
\end{proof}

\subsection{Hermitian  sheaves}

Let $(\cF, H)$ be a Hermitian vector bundle on a complex manifold $X$. 
We  regard $H$ as a selfadjoint conjugate-linear morphism   from  $\cF$  to its dual bundle $\cF^* = \cH om_{\cO_X}(\cF, \cO_X)$.  
We note that the adjoint morphism is the   \textit{conjugate} of  the dual morphism.  
In particular,  for any smooth section $h$ of   $\cE nd {(\cE)}$, for any complex number $\lambda$, we have $(\lambda h)^* = \overline{\lambda}h^*$ for the adjoint operation $*$. 
We note that the trace  
$\mathrm{tr}\, h$ is a well-defined smooth function on $X$.
We can hence define a Hermitian pairing $<\cdot,\cdot>_H$ on  $\cE nd {(\cF)}$ with respect to $H$ as follows. 
For any smooth sections $h$ and $g$ of   $\cE nd {(\cF)}$, we set 
\[ <h,g>_H = \mathrm{tr}\, (h \circ H^{-1} \circ g^* \circ H) 
=  \mathrm{tr}\, (h   H^{-1}   g^*   H).
\]  
We denote the corresponding norm by $|| h ||_H = \sqrt{<h,h>_H}$.  
We say that  $h$   is bounded or integrable if $|| h ||_H $ is.

If $H'$ is another Hermitian metric on $H$, and if we write $h = H^{-1}H'$, then it follows that $h$ is selfadjoint with respect to both $H$ and $H'$. 
That is 
\[ h^*H = H h \mbox{ and } h^*H' = H'h. \] 
In particular, we have 
\[
||h||_{H }  = ||h||_{H'} = \sqrt{\mathrm{tr}\, h^2}.
\]

We will denote  by $\partial_H$ the $(1,0)$-part of the  Chern connection, on $\cF$ or on $\cE nd (\cF)$. 
We  define 
\[ \Delta_H = -2\sqrt{-1}\Lambda \dbar \partial_H. \]
The following computations can be found in \cite[Lemma 3.1]{Simpson1988}

\begin{lemma}\label{lemma:simpson-computation}
Let $(\cF,H)$ be a Hermitian vector bundle on a K\"ahler manifold $(X,\omega)$. 
Assume that $s$ and $h$ are  smooth sections of $\cE nd {(\cF)}$, selfadjoint with respect to $H$. Assume that  $h$ is definite positive. We denote by $F_H$ and  $F_{Hh}$ the corresponding  Chern curvatures. 
The following properties hold.
\begin{enumerate}
\item $|| (\dbar s) h ||_{H,\omega}^2 = -\sqrt{-1} \Lambda \mathrm{tr}\, (\dbar s) h^2 \partial_H s.$
\item $\Delta_H h = 2\sqrt{-1}  h (  \Lambda F_H  - \Lambda F_{Hh} ) - 2\sqrt{-1}\Lambda  (\dbar h)h^{-1} \partial_H h.$
\item $\Delta \log\, (\mathrm{tr}\, h)  \geqslant  -2 ( || \Lambda F_H ||_H + ||\Lambda F_{Hh}||_{Hh}).$
\end{enumerate}
\end{lemma}

Thanks to item (1) of the previous lemma, we say that  $s$ is $L_1^2$ if  $|| (\dbar s)   ||_{H,\omega}^2$ is integrable. 
The following estimate was essentially obtained in \cite[Lemma 2.1]{UhlenbeckYau1986}. 



\begin{lemma}
\label{lemma:h-s-estimate}
Let $(\cF,H)$ be a Hermitian vector bundle on a K\"ahler manifold $(X,\omega)$. 
Let $s$ be a smooth section of $\cE nd(\cF)$, selfadjoint  with respect to $H$. 
We write $h=e^s$.  Assume that $||s||_H \leqslant C$  for some positive constant $C$. 
Then   there are two positive constants $0 <C_1 \leqslant C_2$, depending only on $C$, such that 
\[ C_1 ||\dbar s||_{H,\omega} \leqslant  ||\dbar h||_{H,\omega}  \leqslant C_2 || \dbar s||_{H,\omega}.\]
\end{lemma}

\begin{proof}
Let $(\mathbf{e}_\alpha)$ be a smooth $H$-orthonormal basis of $\cF$, which diagonalizes $s$, and hence $h$. 
Let $(\mathbf{e}_\alpha^*)$ be the dual basis of $\cF^*$. 
Then $h=e^{\lambda_\alpha} \mathbf{e}_\alpha^* \otimes \mathbf{e}_\alpha$, and $s=  \lambda_\alpha  \mathbf{e}_\alpha^* \otimes \mathbf{e}_\alpha$, for real-valued functions $\lambda_\alpha$.  
We write $\dbar \mathbf{e}_\alpha  = A_{\alpha}^\beta \mathbf{e}_\beta$. 
It follows that $\dbar \mathbf{e}_\alpha^*  = - A^{\alpha}_\beta \mathbf{e}_\beta^*$. 
Then we have 
\[   
\dbar h = e^{\lambda_\alpha}\dbar \lambda_\alpha  \mathbf{e}_\alpha^* \otimes \mathbf{e}_\alpha + (e^{\lambda_\alpha} - e^{\lambda_\beta}) A_{\alpha}^\beta \mathbf{e}_\alpha^* \otimes \mathbf{e}_\beta,
\]
and 
\[   
\dbar s =  \dbar \lambda_\alpha  \mathbf{e}_\alpha^* \otimes \mathbf{e}_\alpha + ( {\lambda_\alpha} - {\lambda_\beta}) A_\alpha^\beta \mathbf{e}_\alpha^* \otimes \mathbf{e}_\beta.
\]
Hence 
\[
||\dbar h||_{H,\omega}^2 = e^{2\lambda_\alpha} ||\dbar \lambda_\alpha||^2_\omega + ( e^{\lambda_\alpha} - e^{\lambda_\beta} )^2||A_\alpha^\beta||^2_\omega,
\]
and 
\[
||\dbar s||_{H,\omega}^2 =  ||\dbar \lambda_\alpha||^2_\omega + (  {\lambda_\alpha} -  {\lambda_\beta} )^2||A_\alpha^\beta||^2_\omega,
\]
Since $||s||_H$ is bounded by $C$, the   $|\lambda_\alpha|$ are bounded by some constant $C'$, depending only on $C$.  We also remark that $z\mapsto z^{-1}(e^z-1)$ can be extended to a smooth positive function on $z\in \mathbb{R}$. 
Hence  there are two positive constants $C_1 \leqslant C_2$, depending only on $C$, such that 
\[ C_1 ||\dbar s||_{H,\omega} \leqslant  ||\dbar h||_{H,\omega}  \leqslant C_2 || \dbar s||_{H,\omega}.\]
This completes the proof of the lemma.
\end{proof}

Now we assume that  $(X,\omega)$ is   K\"ahler variety, and $\cF$ a coherent   sheaf on $X$. We will denote by  $$X_\cF \subseteq X$$  the largest open subset  over which $\cF$ is locally free and $X$ is smooth. 
In the remainder of the paper, by a smooth section of $\cF$ on $X$, we mean a smooth section of $\cF$ on $X_\cF$.  
We also call a Hermitian metric $H$ on $\cF|_{X_\cF}$ a Hermitian metric on $\cF$.  
  Let $F_H$ be the Chern curvature of $(\cF,H)$, more precisely, of $(\cF|_{X_{\cF}},H)$. 
  Such a metric is called admissible, after \cite{BandoSiu1994}, if the following two conditions hold. 
  \begin{enumerate}
  \item  $||\Lambda_\omega F_H||_H$ is locally  bounded on $X$.
  \item   $||F_H||_{H,\omega}^2.$ is locally integrable on $X$.
  \end{enumerate}   

In the remainder of the paper, by a smooth section of $\cF$ on $X$, we mean a smooth section of $\cF$ on $X_\cF$.  We also recall the following fact, and borrow the proof from \cite[Section 7]{UhlenbeckYau1986}.

\begin{lemma}\label{lemma:extension-subsheaf} 
Let $X$ be a normal complex variety and $\cF$ a reflexive coherent sheaf. 
Let $X^\circ \subseteq X$ be Zariski open subset whose complement has codimension at least two, and $\cE \subseteq \cF|_{X^\circ}$ a saturated coherent subsheaf. 
Then $\cE$ extends to a coherent saturated subsheaf  of $\cF$, globally on $X$.
\end{lemma}

\begin{proof}
It is enough to prove the extension locally on $X$. 
Hence by embedding $\cF$ in a free coherent sheaf as a saturated subsheaf, we may assume that $\cF$ is free. 
Furthermore, by shrinking $X^\circ$, we may assume that $\cE$ is a subbundle of $\cF|_{X^\circ}$. 
Then we have a morphism $f\colon X^\circ \to  G$, where $G=G(n,m)$ is the Grassmannian variety, with $n=\mathrm{rank}\, \cF$ and  $m=\mathrm{rank}\, \cE$. 
Let $\Gamma \subseteq X\times G$ be the closure of the graph of $f$. 
By applying \cite[Main Theorem]{Siu1975} to normal varieties, as explained in \cite[page 441]{Siu1975},  we obtain that $\Gamma$ is a closed analytic subvariety of $X\times G$.
We denote by $p_1\colon \Gamma \to X$ and $p_2\colon \Gamma \to G$.

Let $U$ be the universal vector bundle on $G$. By pulling back to $\Gamma$ \textit{via} $p_2$, we get a subbundle   $\cG$ of $p_1^*\cF$ on $\Gamma$. 
Since   $p_1$ is proper, the direct image $(p_1)_* \cG$ is a coherent sheaf on $X$.
Furthermore, it is a subsheaf of $\cF$ on $X$, extending $\cE$.
This completes the proof of the lemma.
\end{proof}

\subsection{Slope stability} \label{section:slope}

We will first introduce  two different ways to define the degree of a reflexive sheaf $\cF$  on a compact normal K\"ahler variety $(X,\omega)$.  We note that the definitions rely on the Hermitian metric constructed in Section \ref{section:initial-metric} below.

The first one was considered in \cite[page 877]{Simpson1988}, which \textit{a prior} depends on  the Hermitian metric.  
Let $(M,\omega)$  be a K\"ahler manifold of dimension $n$, 
and $(\cF, H)$ a Hermitian vector bundle on $M$, with Chern curvature $F$.   
Assume that $\mathrm{tr}\, (\Lambda_\omega F)$ is integrable on $(M,\omega)$. 
Then we define the degree of $(\cF, H)$ as 
\[\mathrm{deg}_\omega\, (\cF, H) =  \sqrt{-1}\int_{X} \mathrm{tr}\, (\Lambda_\omega F) \omega^n.\]
If $\cF' \subseteq \cF$ is a non zero saturated coherent analytic subsheaf, then $H$ induces a Hermitian metric on $\cF'$ by restriction.  
By using  Chern-Weil formula, we can define   the degree $\mathrm{deg}\, (\cF', H|_{\cF'})$, which belongs to $[-\infty,  \infty)$.  

We consider the following situation. Let $(X, \omega)$ be a normal compact K\"ahler variety, and $\cF$ a reflexive coherent sheaf on $X$. Let $H$ be a metric constructed by Lemma \ref{lemma:existence:initial-metric}. 
Then $\mathrm{tr}\, (\Lambda_\omega F)$ is integrable on $X$ by construction (see Remark \ref{rem:uniform-integrable}).
Thus we can define $\mathrm{deg}_\omega\, (\cF, H)$ as in the previous paragraph, by taking $M = X_\cF$.

Now we introduce a second   definition of $\mathrm{deg}_\omega\, \cF$.
Let $r\colon \widehat{X} \to X$ be any desingularization such that the torsion-free quotient  $r^*(\cF^*)/{\mathrm{(torsion)}}$ is locally free. 
We then define 
\[  \mathrm{deg}_\omega\, \cF  = -2\pi n c_1 \Big(r^*(\cF^*)/{\mathrm{(torsion)}} \Big) \wedge [r^*\omega]^{n-1},
\] 
where $[r^*\omega]$ is the corresponding cohomology class.  
This definition does not involve metrics. 
It is also independent of the choice of resolution, and appears more intrinsic. To see this, we only need to observe the following fact. If $p\colon W \to \widehat{X}$  is a modification with $W$ smooth, then 
\[p^*\Big( r^*(\cF^*)/{\mathrm{(torsion)}} \Big) \cong \Big((r\circ p)^*(\cF^*)\Big)/{\mathrm{(torsion)}}.
\]
We remark that if $X$ is projective and $\omega$ is induced by an ample divisor, then this definition of degree coincides with the usual one in algebraic geometry.

We claim that $\deg_\omega \cF = \mathrm{deg}_\omega\, (\cF,H)$. 
That is, the previous two definitions coincide.  
Indeed, By Lemma  \ref{lemma:existence:initial-metric}, there is a desingulariztion $r\colon \widehat{X} \to X$ such that $r^*(\cF^*)/{\mathrm{(torsion)}}$ is locally free and $H$ extends to a smooth Hermitian  metric  on $(r^*(\cF^*)/{\mathrm{(torsion)}})^*$. 
This implies that  $\deg_\omega \cF = \mathrm{deg}\, (\cF,H)$. 
We notice that, for any saturated coherent subsheaf $\cF' \subseteq \cF$,  we have
$\mathrm{deg}\, (\cF', H|_{\cF'}) = \mathrm{deg}_\omega \, \cF'$.

We then define the slope of $\cF$, with respect to $\omega$,  as 
\[ \mu_{\omega}(\cF) = \frac{\deg_\omega\, \cF}{\mathrm{rank}\, \cF}.\]
We say that $\cF$ is  {(slope) stable} if for any proper non zero saturated coherent analytic  subsheaf $\cF'\subseteq \cF$, we have  
\[\mu_{\omega}(\cF')  <  \mu_{\omega}(\cF).\]   


\subsection{Hermitian-Yang-Mills metrics}

Let $(X,\omega)$ be a compact normal  K\"ahler variety of dimension $n$,
and  $\cF $   a   reflexive sheaf on $X$.  
A  metric $H$ on $\cF$ is called  Hermitian-Yang-Mills, or  Hermitian-Einstein, if its Chern  curvature $F$ satisfies 
\[ \sqrt{-1}\Lambda_{\omega} F  =  \lambda_\omega(\cF) \mathrm{Id}_\cF, \]
where $\lambda_\omega(\cF) =         \mu_\omega (\cF)\mathrm{Vol}_\omega (X) ^{-1}$.

We fix a Hermitian metric $H_0$ on $\cF$, with Chern curvature $F_0$. 
For any Hermitian metric $H$ on $\cF$, we  write $h=H_0^{-1}H$. 
Then $h$ is smooth section of $\cE nd {(\cF)}$, which is  self-adjoint with respect to   $H_0$ and $H$. 
The curvatures satisfy
\[F_H = F_0 + \dbar( ( h^{-1}\partial_{H_0} h) ).\]  
To obtain a Hermitian-Yang-Mills metric on $\cF$, we  will  consider a family of metrics 
\[
H_t=H_0h(t) = H_0h_t = H_0h\] 
subject to the following heat equation,  with $h_0 = \mathrm{Id}_\cF$,
\begin{equation}\label{equa:heat-equation-metric}
\begin{cases}
h_t^{-1}(\frac{\mathrm{\partial}}{\mathrm{\partial}t}h_t)  = -(\sqrt{-1}\Lambda_{\omega} F_t - \lambda_\omega(\cF) \mathrm{Id}_\cF)\\
\mathrm{det}\, h_t \equiv 1
\end{cases}
\end{equation}
where  $F_t$ is the Chern curvature   of $H_t$.  

\begin{lemma}\label{lemma:parabolic-positive-estimate}
Assume that the heat equation above admits a solution around some time $t>0$. 
Then 
\[(\frac{1}{2}\Delta-\frac{\partial}{\partial t}) || \Lambda F_t ||_{H_t} \geqslant 0\]
whenever $|| \Lambda F_t ||_{H_t} \neq 0.$
\end{lemma}

\begin{proof}
For simplicity, we write $u = \sqrt{-1} \Lambda F_t$.  
The following computation is local and we may hence assume that $u\neq 0$.  
As in the proof of \cite[Lemma 6.1]{Simpson1988}, we  get  
\[  (\frac{1}{2}\Delta-\frac{\partial}{\partial t}) || u ||_{H_t}^2 = 2|| \dbar u ||_{H_t}^2. \] 
We set $m=||u||_{H_t}^2$, then  
\[ \Delta ||u||_{H_t} = \frac{\Delta m}{2 ||u||_{H_t}}  - \frac{\sqrt{-1}}{2||u||_{H_t}^3}  \Lambda \partial m \dbar m,\]  
and 
\[\frac{\partial}{\partial t} ||u||_{H_t} = \frac{1}{2|| u ||_{H_t}}\frac{\partial}{\partial t} m.    \]
We thus deduce that 
\[    (\frac{1}{2}\Delta-\frac{\partial}{\partial t}) || u||_{H_t}  =  \frac{|| \dbar u||_{H_t}^2}{||u||_{H_t}}   -  \frac{1}{4|| u ||_{H_{t}}^3} ||\dbar m||^2. 
\]
We note that 
\[ || \dbar m||^2 =4\sqrt{-1} \Lambda \mathrm{tr}\, (u\dbar u) \mathrm{tr}\, (u\partial_{H_t} u).\]
Since $u$ is selfadjoint with respect to $H_t$, by taking a local $H_t$-orthonormal basis of $\cF$ which diagonalizes $u$, we see that $4||\dbar u||_{H_t}^2||u||_{H_t}^2 \geqslant ||\dbar m||^2$. 
Hence  we get
$(\frac{1}{2}\Delta-\frac{\partial}{\partial t}) || u||_{H_t} \geqslant 0$. 
This completes the proof of the lemma.
\end{proof}

\section{Construction of an initial metric}
\label{section:initial-metric}

The goal of this section is to construct an initial metric $H_0$ for the heat equation (\ref{equa:heat-equation-metric}), see Lemma \ref{lemma:existence:initial-metric} below.  
Such a metric will guarantee that $\Lambda F$ is integrable, where $F$ is the Chern curvature.  
We will follow the lines of \cite[Section 3]{BandoSiu1994}. 
The idea is to construct a metric which is locally the restriction of a smooth metric on a vector bundle.
The first step is to construct an appropriate partition of the unity.

\begin{prop}\label{prop:partition}
Let $X$ be an analytic variety. Let $\{U_i\}$ be a finite open covering of $X$. We assume that each $U_i$ can be identified as a closed analytic subvariety of a domain $V_i \subseteq \mathbb{C}^{N_i}$.
Then there is a partition of the unity, subordinate to $ \{U_i\}$, which satisfies the following property.
For  $x\in X$, there is an open neighbourhood $U$ of $x$, such that 
\begin{enumerate}
\item $U$ is contained in $U_i$ whenever $x\in U_i$,
\item if $U \subseteq U_i$, then there is an open subset $W_i\subseteq V_i$ containing $U$ such that  $\rho_i|_U$ is the restriction of some smooth function $\eta_i\colon W_i \to \mathbb{R}$. 
\end{enumerate}
\end{prop}

\begin{proof}
Since a partition of the unity  can be constructed by  taking linear combinations of  products of  bump functions, we will construct bump functions on $U_i$ which satisfy the property in the proposition.

Without loss of generality, we take a bump function $\gamma_1$ on $V_1$, which restricts to a bump function $\rho_1$  on $U_1$. Let $x\in U_1$ be any point. 
Assume that $x\in U_i$. 
Then  there exists open subsets  $x\in W_1\subseteq V_1 \subseteq \mathbb{C}^{N_1}$  and $x\in W_i\subseteq V_i \subseteq \mathbb{C}^{N_i}$,  such that the identity map of $U_i\cap U_1$ induces a transition function
\[ \varphi_{i,1} \colon W_i \to W_1, \] which is holomorphic. 
We set $\eta_i = \gamma_1 \circ \varphi_{i,1}$. 
Let  $U$ be an open neighbourhood $x$ which is contained both in $W_1$ and $W_i$.  
Then $\rho_1|_U = \eta_i|_U$. 
By taking account of all $U_i$, we complete the proof of the lemma.
%
%
\end{proof}

We will also need the following local estimate.

\begin{lemma}
\label{lemma:curvature-local-bound}
Let $(X, \omega)$ be a    K\"ahler subvariety  of a K\"ahler manifold $(V,\omega)$. 
Let  \[ \cF \subseteq \cO_X^a \]  be a saturated coherent subsheaf.  
Let $g $ be a Hermitian metric on $\cO_V^a$ and $h$ its restriction on $\cO_X^a$.
We denote by $h_\cF$ the restriction of $h$  on $\cF$. 
Let $F_{h_\cF}$ be  the Chern curvature. 
Then   for any precompact open subset $U$ of $X$, there is  a constant $C>0$ such that 
\[ \sqrt{-1} F  \leqslant C\omega \cdot \mathrm{Id}_\cF\] 
\end{lemma}

\begin{proof}
We denote   by  $F_h$    and  $F_g$ the corresponding Chern curvatures. 
Then we have  
\[F_h = F_g|_X.\]
Furthermore, we denote by   $\nabla_h$ and $\nabla_{h_{\cF}}$ the  corresponding Chern connections. 
Then there is a   $(1,0)$-form $A$ on $X$, with values in $\cH om(\cF, \cF^{\perp_h })$, such that,  for any smooth section $s$ of $\cF$ on $X$,  we have 
\[  \nabla_h s = \nabla_{h_\cF} s+ As.\]
Then $$ \sqrt{-1} F_{h_\cF} = \sqrt{-1}F_{h}|_{\cF} - \sqrt{-1}A\wedge A^* \leqslant \sqrt{-1}F_h|_{\cF}.$$  

For a precompact open set $U$, then there is a precompact open subset $W\subseteq V$, which contains $U$. 
Then   there is some constant  $C$,   such that,    on $W$,  we have
\[   \sqrt{-1} F_g   \leqslant C\omega \cdot \mathrm{Id}_{\cO_V^{a}}.  \]
Thus  $ \sqrt{-1} F_{h_\cF}   \leqslant C\omega \cdot \mathrm{Id}_\cF$ on $U$. 
This completes the proof of the lemma.
\end{proof}

Now we are ready to construct the desire metric.

\begin{lemma}\label{lemma:existence:initial-metric}
Let $(X, \omega)$ be a normal   compact K\"ahler variety and $\cF$ a coherent reflexive sheaf on $X$. 
Then there exists a Hermitian metric $H$ on $\cF$ with Chern curvature $F$ such that  
\[ \sqrt{-1} F  \leqslant C\omega \cdot \mathrm{Id}_\cF\] 
for some positive constant $C$.    

Moreover, let $r\colon \widehat{X} \to X$ be a desingularization  by successively blowing up smooth centers. 
Assume  that the torsion-free quotient $  r^*(\cF^*) /({\mathrm{torsion}})$ is locally free, then $H$ extends to a smooth Hermitian   metric $\widehat{H}$ on the locally free sheaf $ (r^*(\cF^*))^*$.
\end{lemma}

\begin{proof} 
There is a finite open covering $\{U_i\}$ of $X$ such that the following properties hold.
\begin{enumerate}
\item Every $U_i$ is precompact and can be identified with an analytic subvariety of a domain $V_i \subseteq \mathbb{C}^{N_i}$.

\item There is a surjective morphism 
\[  \cE_i \to \cF^*|_{U_i} \to 0
\] 
such that   $ \cE_i$ is a  free coherent sheaf on $U_i$.   
\end{enumerate}
Let $\cE_i^{V_i}$ be a free coherent sheaf on $V_i$ such that $\cE_i = \cE_i^{V_i}|_{U_i}$. 
Let $g_i$ be a Hermitian metric  on  $(\cE^{V_i})^*$.  
By restriction, it  induces a Hermitian metric on $h_i$ on  $\cF|_{U_i}$.  
Let $\{\rho_i\}$ be a partition of the unity subordinate to $\{U_i\}$, as in Proposition \ref{prop:partition}. 
We define the Hermitian metric $H = \sum \rho_i h_i$ on $\cF$.  
We denote by  $F$   its Chern curvature.
We will show that $H$ has the properties of the lemma.

Let $x \in X$ be a point. Then there is an open neighbourhood $x \in U$  such that 
\begin{enumerate}
\item $U$ is contained in $U_i$ whenever $x\in U_i$.
\item if $x\in U \subseteq U_i$, then there is an open subset $W_i\subseteq V_i$ containing $U$ such that  $\rho_i|_U$ is the restriction of some smooth function $\eta_i\colon W_i \to \mathbb{R}$. 
\end{enumerate}
Now we consider the diagonal embedding, followed by the product inclusion,  
\[U  {\hookrightarrow}    \prod_{x\in U_i} U_i  \hookrightarrow   \prod_{x\in U_i} W_i. \] 
Let $\iota \colon U \to \prod_{x\in U_i}W_i$ be the composition of this sequence. 
Then it realizes $U$ as a K\"ahler   subvariety of the manifold $\prod_{x\in U_i}W_i,$ up to scaling the K\"ahler form by the cardinality of $\{ U_i\ |\    x \in U_i \}$.

We denote by $p_i\colon\prod_{x\in U_i} W_i \to W_i $ the natural projections. 
By abuse of notation, we still denote by $p_i$ the projection from $\iota(U)$ to $U\subseteq U_i \subseteq W_i$. 
For the reflexive sheaf $\cF|_U$, we have the following sequence of sheaves supported in  $ \iota(U) \subseteq \prod_{x\in U_i} W_i $, such that the first one is the diagonal map, and the second one is the direct sum embedding, 
\[ \iota_*(\cF|_U) \hookrightarrow  \bigoplus_i p_i^*(\cF|_U)  \hookrightarrow   \bigoplus_i  p_i^*(  \cE_i^*|_U).\]

We remark that 
\[\bigoplus_i p_i^* ( \cE_i^*|_U )=   \Big(\bigoplus_i p_i^* ((\cE^{V_i})^*|_{W_i} ) \Big) \Big|_{\iota(U)}. \] 
The vector bundle $ \bigoplus_i p_i^* ((\cE^{V_i})^*|_{W_i} ) $ can be equipped with the orthogonal direct sum metric $g' = \sum \eta_i p_i^*g_i$. 
Then the metric $H$ on $\cF|_U$ coincides exactly the restriction of $g'$ on $\iota_*(\cF|_U)$.
Hence by Lemma \ref{lemma:curvature-local-bound}, we have 
\[\sqrt{-1} F   \leqslant  C_x \omega \cdot \mathrm{Id}_\cF\] 
over some open neighbourhood of $x$, for some positive number $C_x$. 
 
Since $X$ is compact, it follows that there is a positive constant $C$ such that 
\[\sqrt{-1} F   \leqslant  C \omega \cdot \mathrm{Id}_\cF.\]

For the second part of the lemma,    we will first  work locally. 
Let $\widehat{U}_i = r^{-1}(U_i)$. 
Since $r$ is obtained by blowing smooth centers, we can blow up $V_i$ with the same centers and obtain a manifold $q\colon \widehat{V}_i \to V_i$, such that $\widehat{U}_i$ is the strict transform of $U_i$ in $\widehat{V}_i$. 
We have the following exact sequence, 
 \[  r^*\cE_i \to  r^*\cF^*|_{\widehat{U}_i} \to 0. \]
By taking the dual morphism, we see that  $ (r^*\cF^*)^* |_{\widehat{U}_i}$ is a subsheaf of $ (r^*\cE_i)^*$. 
By the  assumption of freeness,  $ (r^*(\cF^*))^*|_{\widehat{U}_i}$ is  indeed a subbundle of $  (r^*\cE_i)^* $.
The metric $q^*g_i$ on $q^* (\cE^{V_i})^*$  then  induces a metric $\widehat{h}_i$ on $ (r^*\cF^*)^*|_{\widehat{U}_i}$ by restriction. 
We note that $\widehat{h}_i$ is an extension of $h_i$. 
By taking account of all $U_i$, we see that $H$ extends to a smooth Hermitian metric  $\widehat{H}$ on the locally free sheaf $( r^*\cF^*)^*$.
\end{proof}

For such a metric, we have the following statement on integrability.

\begin{cor}
\label{cor:initial-metric-integral}
With the notation above, let  $n=\dim X$.  
Let $Z \subseteq X$ be the   locus where $r^{-1}$ is not an isomorphism. 
For any K\"ahler form $\eta$ on $\widehat{X}$, we still denote by $\eta$ its restriction on $X\backslash Z$. 
Then on $X\backslash Z$, we have 
\[ ||\sqrt{-1}  \Lambda_\eta F  ||_H  \eta^{n } \leqslant \mathrm{tr}\, \Big( 2nC \omega \mathrm{Id}_\cF   - n \mathrm{tr}\, (\sqrt{-1}F ) \Big) \wedge \eta^{n-1}.  \]
As a consequence,  
\begin{eqnarray*}
\int_{{X\backslash Z}} ||\sqrt{-1}  \Lambda_\eta F  ||_H  \eta^{n }  &\leqslant&  2nC \cdot \mathrm{rank}\, \cF\cdot [r^*\omega]\wedge [\eta^{n-1}]\\
&& -2n\pi \cdot c_1\Big( (r^*(\cF^*))^* \Big) \wedge  [\eta^{n-1}] 
\end{eqnarray*}
is bounded, where $[\eta]$ and $[r^*\omega]$ are the corresponding cohomology classes on $\widehat{X}$.
\end{cor}

\begin{proof}
We have  
\[ ||\sqrt{-1}  \Lambda_\eta F  ||_H \leqslant ||\sqrt{-1}  \Lambda_\eta(C\omega \mathrm{Id}_{\cF} - F)  ||_H  + ||C\sqrt{-1}  \Lambda_\eta \omega \mathrm{Id}_{\cF} ||_H.
\]
For simplicity, we set $g= \sqrt{-1}  \Lambda_\eta(C\omega \mathrm{Id}_{\cF} - F). $
Since $F$ is the Chern curvature of $H$, we see that  $g$ is selfadjoint with respect to  $H$. 
Thus $ ||g ||_H = \sqrt{\mathrm{tr}\, (g^2)}$.
Since $g$  is positive by Lemma \ref{lemma:existence:initial-metric}, it follows that 
\[ ||g ||_H  \leq \mathrm{tr}\, g = \sqrt{-1} \Lambda_\eta\mathrm{tr}\, (  (C\omega \mathrm{Id}_{\cF} - F)). \] 
Thus \[ 
||\sqrt{-1}  \Lambda_\eta F  ||_H  \eta^{n } \leqslant  n \mathrm{tr}\, (2 C\omega \mathrm{Id}_{\cF} - F) ) \eta^{n-1}.  \]
This proves the first inequality. 

For the estimate on the integral, we note that $F$ can be viewed as the restriction of $\widehat{F}$ on $X\backslash Z$, where $\widehat{F}$ is the Chern connection of $\widehat{H}$ on $( r^*\cF^*)^*$.   
It follows that 
\[ \int_{X\backslash Z} \mathrm{tr}\, (\sqrt{-1} F)\wedge \eta^{n-1} = \int_{\widehat{X}} \mathrm{tr}\, (\sqrt{-1} \widehat{F} )\wedge \eta^{n-1} =2\pi c_1\Big( (r^*(\cF^*))^* \Big) \wedge  [\eta^{n-1}]. \]
This completes the proof.
\end{proof}

\begin{remark}\label{rem:uniform-integrable}
With the notation above, we assume that $\omega_\epsilon$ is a family of K\"ahler metric on $X\backslash Z$, bounded from above by $\eta$. 
Then the corollary indeed implies that 
$\int_{{X\backslash Z}} ||\Lambda_{\omega_\epsilon} F  ||_H  \omega_\epsilon^{n }$
are uniformly integrable, independent of $\epsilon$,  in the following sense. 
We fix an exhaustion $\{X_j\}$ of $X\backslash Z$, consisting of precompact subsets.  
Let  $T_j = {X\backslash (Z \cup X_j) }$.
Then for each $\xi >0$, there is a positive number $m_0$ such that for any $j \geqslant m_0$, we have  
\[
\int_{T_j} || \Lambda_{\omega_\epsilon} F  ||_H  \omega_\epsilon^{n } \leqslant \xi.
\]
Indeed,  as in Corollary \ref{cor:initial-metric-integral}, we have $ || \Lambda_{\omega_\epsilon} F  ||_H  \omega_\epsilon^{n } \leqslant \Theta \wedge \omega_\epsilon^{n-1},$ where we denote $\Theta = \mathrm{tr}\, \Big( 2nC \omega \mathrm{Id}_\cF   - n \mathrm{tr}\, (\sqrt{-1}F ) \Big)$.  
Lemma \ref{lemma:existence:initial-metric} implies that $\Theta$ is positive. 
Hence $0 \leqslant \Theta \wedge \omega{_\epsilon}^{n-1} \leqslant \Theta \wedge \eta^{n-1}.$ 
Since $\Theta\wedge \eta^{n-1}$ is integrable, there is some $m_0$ such that for any $j \geqslant m_0$, we have 
\[
\int_{T_j}  \Theta \wedge \eta^{ n-1}  \leqslant \xi.
\]
This implies the estimate. 
\end{remark}

\section{Heat kernels of K\"ahler varieties}
\label{section:heat-kernel}

In this section, we will investigate the heat kernels of   compact K\"ahler varieties. 
We will start by recalling some basic properties on heat kernels.

\subsection{Heat kernesl of Riemannian manifolds}
Let $(M,g)$ be a Riemannian manifold, which is not necessarily compact.  
We denote by $\Delta$ the Laplace-Beltrami operator induced by  $g$. 
Let $\Omega\subseteq M$ be any non empty precompact open subset, with possibly empty smooth boundary $\partial\Omega$.  
We consider the following heat equation on $\Omega$, with Dirichlet boundary condition, on $u(x,t)$, with   variables  $x\in \Omega$ and $t\in [0,\infty)$,
\begin{equation}\label{equa:heat-equation-precampact}
    \begin{cases}
       (\Delta-\frac{\partial}{\partial t}) u(x,t) = 0 & \mbox{for } t>0\\
       u(x,0) =   f(x)    &\\
       u(x,t)|_{\partial \Omega} = 0 & \mbox{for } t>0
    \end{cases}       
\end{equation}
The initial condition $f$ is assume to be  bounded continuous. 
This equation admits  a smooth fundamental solution $K_\Omega(x,y,t)$, which is called the heat kernel of $(\Omega,g)$, in the following sense. 
One solution of the heat equation (\ref{equa:heat-equation-precampact})  is  the convolution
\[ u(x,t) = \int_{y\in \Omega} K_\Omega(x,y,t)f(y) \mathrm{d}y.\]


We now pass to the following  heat equation on $M$, with $f$ continuous and bounded, 
\begin{equation}\label{equa:heat-equation-open}
    \begin{cases}
       (\Delta-\frac{\partial}{\partial t}) u(x,t) = 0 & \mbox{for } t>0\\
       u(x,0)  =  f(x)& \\
    \end{cases}       
\end{equation}
Such an equation also admits a smooth fundamental solution in the previous sense. 
One can be obtained as follows (see \cite[Theorem 3.6]{Dodziuk1983}).  
Let $\{\Omega_i\}$ be an exhaustion  of $X$ by precompact open subsets, and  we set $$K(x,y,t) = \limsup_i K_{\Omega_i}(x,y,t).$$ 
Then it is  the smallest positive fundamental solution, and we call it the heat kernel of $(X,g)$.

We have the following properties on the heat kernels $K$ and $K_{\Omega}$.  
For simplicity, we omit the lower index $\Omega$ for $K_{\Omega}$ in the following list.  
\begin{enumerate}
\item $K(x, y, t) \geqslant 0$.
\item $K(x,y,t) = K(y,x,t)$.
\item $K (x,y,t+s) = \int  K(x,z,t)K(z,y,s) \mathrm{d}z$.
\item $K(x,y,t) \leqslant  \sqrt{K(x,x,t)K(y,y,t)}.$
\item $  K(x,z,t)  \mathrm{d}z  \leqslant 1$.
\end{enumerate}

We can derive the following estimate after the maximum principal. 
For the reader's convenience, we recall the proof here.

\begin{prop}\label{prop:max-principal}\
Let $(M,g)$ be a compact Riemannian manifold without boundary, and  $K$ its heat kernel.
Let $v(x,t)\geqslant 0$ be a continuous  function on $M\times [0,\infty)$.
Assume that, whenever $v(x,t)\neq 0$, it is $\mathcal{C}^2$ in $x$, $\mathcal{C}^1$ in $t$, and  
\[(\frac{1}{2}\Delta-\frac{\partial}{\partial t}) v \geqslant 0.\] 
Then for any $t>0$, we have 
\[ v(x,t) \leqslant \int_{y\in M} K(x,y,\frac{1}{2}t)v(y,0)\mathrm{d}y.\]
\end{prop}

\begin{proof}
We set  
\[u(x,t) =\int_{y\in M} K(x,y,\frac{1}{2} t)v(y,0)\mathrm{d}y \geqslant 0. \] 
Then $u$ is a smooth solution of the   equation $(\frac{1}{2}\Delta-  \frac{\partial}{\partial t}) u =0$, with initial data $u(x,0)=v(x,0)$.   
Let $w = v-u$, then \[(\frac{1}{2}\Delta-  \frac{\partial}{\partial t}) w \geqslant 0. \]
We will show that $w\leqslant 0$ over $M\times [0,\infty)$. 
Assume by contradiction that $w$ admits strictly positive values at some time $T>0$.  
For some  $0<\delta \ll 1$, we set 
\[w_\delta (x,t) = w(x,t) - \delta t.\]
Then $w_\delta$  also admits  strictly positive values at $T$. 
By continuity of $w_\delta$, it admits a maximum at some point $(x_0,t_0) \in M\times [0,T]$.
Since $w_\delta(x_0,t_0)>0$, we get $t_0>0$. 
Moreover, both $ \Delta w_\delta $ and $\frac{\partial}{\partial t} w_\delta$ are  well defined. 

Since $w_\delta$ has maximum at $(x_0,t_0)$, we get 
\[ \Delta w_\delta (x_0,t_0) \leqslant 0.\]   
Therefore, 
\[  \frac{\partial}{\partial t}  w_\delta(x_0,t_0)  \leqslant -\delta + \frac{1}{2}\Delta w_\delta(x_0,t_0) \leqslant -\delta.\] 
It follows that $w_\delta(x_0,\cdot)$ is strictly decreasing around $t_0$. 
Since $t_0 \neq 0$, this  contradicts that  $w_\delta$ has maximum at $(x_0,t_0)$.
\end{proof}

We will show the next two results on upper bounds for the heat kernels.

\begin{prop}[{\cite[Theorem 2.1]{CarlenKusuokaStroock1987}}]\label{prop:sobolev-heat-estimate}
Let $K$ be the heat kernel of a Riemannian manifold $(M,g)$ of real dimension $m$.
Then for any  $t\geqslant t'>0$, we have   $K(x,x,2t) \leqslant K(x,x,2t')$. 
Assume  further  that   we have a Sobolev type inequality  as follows. 
For any compactly supported $\mathcal{C}^1$  function  $f\geqslant 0$,  
\[ (\int_{M} |f|^{\frac{2m}{ m-2} })^{\frac{m-2}{m}}  \leqslant C_S \int_{M} (|\nabla f|^2 + |f|^2). \]
Then   $ K(x,x,2t) \leqslant \Big(\frac{ mC_S}{4} \Big)^{\frac{m}{2}} e^{2t} t^{-\frac{m}{2}}$. 
\end{prop}

\begin{proof}
We note that $K = \limsup K_{\Omega_i}$, where $\{ \Omega_i \}$ is an exhaustion of $M$ consisting of precompact open subsets, with smooth boundary. 
In order to prove the proposition, it is enough to prove the same statement for each $K_{\Omega_i}$. 
For simplicity, we omit the index $\Omega_i$ in the following argument, and assume that $K$ is one of the $K_{\Omega_i}$. In particular, it has compact support in $M  \times M$.

Let $J(x,t)= \int_{y\in M} K(x,y,t)^2 \mathrm{d}y$.  Then  
\[  \frac{\partial}{\partial t} J(x,t) =  2   \int_{y\in M} K (x,y,t) \Delta_y K (x,y,t) \mathrm{d}y = -  2\int_{y\in M}|\nabla_y K(x,y, t)|^2 \mathrm{d}y.\] 
In particular, we see that $J(x,t) \leqslant J(x,t')$ if $t\geqslant t'$.

The Sobolev inequality implies that 
\[ \int_{y\in M}|\nabla_y K(x,y, t)|^2 \mathrm{d}y  \geqslant \frac{1}{C_S} \Big( \int_{y\in M} K(x,y,t)^{\frac{2m}{ m-2}}\mathrm{d}y \Big)^{\frac{m-2}{ m}} - \int_{y\in M} K(x,y,t)^2\mathrm{d}y. \] 
The H\"older inequality implies that 
\[ \int_{y\in M} K(x,y,t)^{\frac{2m}{ m-2}}\mathrm{d}y  \geqslant  \Big(\int_{y\in M} K(x,y,t)^2\mathrm{d}y  \Big)^{\frac{ m+2}{ m-2}} \Big( \int_{y\in M} K(x,y,t)\mathrm{d}y \Big)^{-\frac{4}{ m-2}}.  \]
We note that 
\[ \int_{y\in M} K(x,y,t)\mathrm{d}y  \leqslant 1, \]
thus we have 
\[ \int_{y\in M}|\nabla_y K(x,y, t)|^2 \mathrm{d}y  \geqslant \frac{1}{C_S} J(x,t)^\frac{m+2}{m} - J(x,t). \]
Therefore 
\[ -\frac{\partial}{\partial t} J(x,t) \geqslant \frac{2}{C_S}J(x,t)^{\frac{m+2}{m} } - 2J(x,t).  \]

Let $I(x,t)=J(x,t)e^{-2t}$. 
It follows that 
\begin{eqnarray*}
-\frac{\partial}{\partial t}I(x,t) 
& = & e^{-2t}(\frac{\partial}{\partial t}J(x,t) -2J(x,t) )\\
& \geqslant &\frac{2e^{-2t}}{C_S} J(x,t)^{\frac{m+2}{m}} \\
& \geqslant & \frac{2}{C_S} I(x,t)^{\frac{m+2}{m}}.
\end{eqnarray*}
We remark that $I(x,t)$ tends to infinity when $t$ tends to zero. 
Hence we have  $I(x,t) \leqslant \Big(\frac{ mC_S}{4} \Big)^{\frac{m}{2}} t^{-\frac{m}{2}}$, and thus  $J(x,t) \leqslant \Big(\frac{ mC_S}{4} \Big)^{\frac{m}{2}} e^{2t} t^{-\frac{m}{2}}$. 
\end{proof}

Following the method of \cite{Grigoryan1997}, we obtain a Gaussian upper bound for the heat kernel.

\begin{cor}\label{cor:sobolev-heat-estimate-2points}  
With the assumptions in Proposition \ref{prop:sobolev-heat-estimate}, we have the following estimate for $K(x,y,t)$. There are constants $t_0 >0$, $\mu>0$, depending only on $m$ such that 
\[ K(x,y,t) \leqslant \frac{2^{m+2}e^{t_0}}{\sigma(\mu t)} \exp \Big( -\frac{r^2(x,y)}{5t} \Big), \]
where $r(x,y)$ is the distance between $x$ and $y$, and the function $\sigma$ is defined by 
\begin{equation*} 
\sigma(t) = 
    \begin{cases}
       \Big(\frac{2}{ mC_S} \Big)^{\frac{m}{2}} e^{-t} t^{\frac{m}{2}} & \mbox{for } 0 < t<t_0\\
        \Big(\frac{2}{ mC_S} \Big)^{\frac{m}{2}} e^{-t_0} t_0^{\frac{m}{2}} & \mbox{for } \mbox{for } t \geqslant t_0 \\
    \end{cases}       
\end{equation*} 
\end{cor}

\begin{proof}
By Proposition \ref{prop:sobolev-heat-estimate}, we see that $ K(x,x,t) \leqslant \Big(\frac{ mC_S}{2} \Big)^{\frac{m}{2}} e^{t} t^{-\frac{m}{2}}$ for all $t>0$. 
There is some $t_0>0$, depending only on $m$, such that $e^{-t}t^{\frac{m}{2}}$ is  increasing on $t\in (0,t_0)$.   
Hence the function  $\sigma$ is a regular function in the sense of \cite[page 37]{Grigoryan1997}, with $\gamma = 2$ and $A = 2^{m}e^{t_0}$. 
Moreover, $K(x,x,t) \leqslant \frac{1}{\sigma(t)}$ for all $t>0$ since $K(x,x,t)$ is decreasing in $t$ by Proposition \ref{prop:sobolev-heat-estimate}. 
Hence we can apply \cite[Theorem 3.1]{Grigoryan1997}, and deduce that, there is a  constant   $\mu>0$, depending  on $m$ (and $t_0$)  such that 
\[ K(x,y,t) \leqslant \frac{2^{m+2}e^{t_0}}{\sigma(\mu t)} \exp \Big( -\frac{r^2(x,y)}{5t} \Big). \] 
This completes the proof of the corollary.
\end{proof}

For further study on the heat kernels, we will need the following parabolic interior Schauder estimate. For more details, see for example \cite[Theorem 8.12.1]{Krylov1996}.

\begin{lemma} 
\label{lemma:heat-kernel-local-uniform-bound}
Let $B_{R} \subseteq \mathbb{R}^N$ be the ball of radius $R$ with center at the origin.
Let $\Delta$ be a linear  operator  of the following shape,
\[\Delta = a^{ij}(x,t)\partial_i\partial_j+ b^{i}(x,y)\partial_i + c(x,t), \]
where $a^{ij}=a^{ji}$,  $x \in \mathbb{R}^N$ is the spatial parameter and $t$ is the time parameter.
Assume that $u(x,t)$ is a bounded smooth function on $B_{2R}\times (0,2R)$,  such that 
\[ (\Delta - \frac{\partial}{\partial t}) u = 0.\]
Assume that there is a positive number $\lambda$ such that  
\[ - \lambda |\xi|^2 \leqslant  \sum_{i,j} a^{ij}(x,t) \xi_i\xi_j \leqslant   \lambda |\xi|^2 \]
for any $\xi\in \mathbb{R}^N$,  $x\in B_{2R}$ and $t\in [0,2R]$.
Let $k\geqslant 2$ be an integer.
Assume furthermore that all the partial derivatives of the functions $a^{ij}$,$b^i$,$c$, of order at most $k+1$ with respect  to the spatial parameter, of order at most $2$ with respect to the time parameter, 
are bounded by $\lambda$ over $B_{2R}$.
Then there is a constant $C$, depending only on $N, R,k,\lambda$, such that 
\[||\partial^s u||_{\infty, B_{R}\times (0,R)} \leqslant C ||u||_{\infty, B_{2R}\times [0,2R]},\]
for any partial derivative $\partial^s u$,  of order at most $k$, with respect  to the spatial parameter.
\end{lemma}

\subsection{Heat kernels of compact K\"ahler varieties}
We will study the heat kernel of a Zariski dense open subset of a compact K\"ahler variety.  
We first prove a Sobolev type inequality.

\begin{lemma}\label{lem:Sobolev-prop-single}
Let $(X,\omega)$ be a compact K\"ahler variety of dimension $n$. 
Then there exists a constant $C_S$,    such that for any $\mathcal{C}^1$ function $f \geqslant 0$, compactly supported in the smooth locus of $X$,   the following inequality holds,
\[ (\int_{X} |f|^{\frac{2n}{ n-1} }\omega^n)^{\frac{ n-1}{ n}}  \leqslant C_S \int_{X} (||\nabla f||^2 + |f|^2) \omega^n. \]
\end{lemma}

\begin{proof}
Since $X$ is compact, by using a partition of the unity, it is enough to prove the problem locally on $X$. 
Hence we may assume that $X$ is a subvariety of $\mathbb{B}^p$, where $\mathbb{B}^p$  is   the unit balls of $\mathbb{C}^p$ and  $\mathbb{C}^q$ respectively.

We note that, for two K\"ahler forms $\omega_a, \omega_b$ on $X$ with $0 < c_1 \omega_a \leqslant \omega_b \leqslant c_2 \omega_a$, if $C_a$ is a Sobolev constant for $\omega_a$, then $(\frac{c_2+1}{c_1})^n C_a$ is a Sobolev constant for $\omega_b$. 
Therefore, we may assume that the K\"ahler form $\omega$ on $X$ is the restriction of the Euclidean K\"ahler form on $\widehat{\mathbb{B}}^p$. 
Then  $X$   is  a minimal subvariety of $\mathbb{B}^p$  after \cite[Theorem 3.1.2]{Simons1968}. 
We can then conclude by applying \cite[Theorem 2.1]{MichaelSimon1973}.
\end{proof}

We can then deduce the following result, which extends  \cite[Lemma 3.1]{LiTian1995}.   
It includes Theorem \ref{thm:heat-kernel-conservation} in the introduction.

 \begin{cor}
\label{cor:heat-kernel-conservation}
Let $X^\circ \subseteq X$ be a  smooth Zariski open dense subset of  a compact K\"ahler variety  $(X,\omega)$, and $K $ the heat kernel of $(X^\circ, \omega)$.  
Then the following properties hold.
\begin{enumerate}
\item $K$ is bounded on $X^\circ \times X^\circ \times [t_1,t_2]$ for any fixed $0<t_1<t_2$. 
\item $K (x,\cdot,t)$ belongs to $L_1^2$.
\item \[\int_{y\in X^\circ}|\Delta_y K (x,y, t)|^2 \mathrm{d}y < \infty.\]
\item \[ 
\int_{y\in X^\circ}K (x,y,t) \mathrm{d}y = 1.
\] 
\end{enumerate}
\end{cor} 
  
\begin{proof} 
By Lemma \ref{lem:Sobolev-prop-single}, a Sobolev type inequality holds on $X^\circ$. 
Hence we get the boundedness of $K$ by applying Corollary \ref{cor:sobolev-heat-estimate-2points}.

The remainder part of the corollary follows from the argument of \cite[Lemma 3.1]{LiTian1995}. 
Indeed,  the assumption in  \cite[Lemma 3.1]{LiTian1995} that $X$ is a closed algebraic subvariety of $\mathbb{P}^N$ serves to get  the boundedness of $K$  in  \cite[Theorem 2.1]{LiTian1995}. 
This is item (1) we proved in the previous  paragraph.  
\end{proof}

The last  property in the previous corollary is also called the conservation property. It implies   that the heat kernels are the same for Zariski open dense subsets.

\begin{cor}\label{cor:heat-kernel-restriction}
With the assumption of Corollary \ref{cor:heat-kernel-conservation}, if $X'\subseteq X^\circ$
is  a Zariski open dense subset. Then the heat kernel $K_{X'}$ of $X'$  is equal to the restriction of $K$, the heat kernel  of $X^\circ$, on   $X'\times X' \times (0,\infty)$. 
\end{cor}

\begin{proof}
From the construction of  heat kernels, we see that $K \geqslant K_{X'}$ over $X'\times X' \times (0,\infty)$. 
By applying Corollary \ref{cor:heat-kernel-conservation} to $K_{X'}$, we get 
\[ 
\int_{y\in X'}K_{X'} (x,y,t) \mathrm{d}y = 1 =  
\int_{y\in X'}K (x,y,t) \mathrm{d}y.
\]
Hence $K_{X'}=K$ on $X'\times X' \times (0,\infty)$. 
\end{proof}

Thanks to this lemma, we will not distinguish the heat kernels $K_{X'}$ and $K$ in the remainder of the paper.

We would like to thank Henri Guenancia and Hans-Joachim Hein for    pointing out that Corollary \ref{cor:heat-kernel-conservation} implies \cite[Conjecture 3.1]{DiNezzaGuedjGuenancia2020}. 
To end this section, we will explain it in the following corollary.
As a consequence, \cite[Theorem F]{DiNezzaGuedjGuenancia2020} still holds without assuming Assumption 3.2 therein.

\begin{cor}
\label{cor:DNGG}
Let $(\mathcal{X},\omega)$ be a K\"ahler space and $\pi\colon \cX \to \mathbb{D}$ a proper holomorphic fibration over the unit disk. 
Let $\Theta$  be a smooth closed $(1,1)$-form on $\mathcal{X}$. 
Assume that every fiber ${X}_t$ of $\pi$ is reduced and irreducible of dimension $n$.  
Then there exists a constant $C>0$ such that 
\[
\sup_{X_t} \varphi_t-C \leqslant \frac{1}{V} \int_{X_t} \varphi_t\omega_t^n \leqslant \sup_{X_t} \varphi_t,
\]
for all $t\in \mathbb{D}_{\frac{1}{2}}$ and $\varphi_t \in \mathrm{PSH}(X_t,\theta_t)$. 
Here $\mathbb{D}_{\frac{1}{2}}$ is the disk of radius $\frac{1}{2}$,   $\theta_t$ is the restriction $\Theta|_{X_t}$, and $V$ is the volume of $X_t$.
\end{cor}

\begin{proof}
Thanks to Corollary \ref{cor:heat-kernel-conservation}, we may prove \cite[Lemma 3.11]{DiNezzaGuedjGuenancia2020} for any compact K\"ahler variety. 
Then we can apply the argument of \cite[Section 3.7]{DiNezzaGuedjGuenancia2020} to conclude the corollary.
\end{proof}

\subsection{Conjecture on uniform Sobolev inequalities}

To apply the method of Bando-Siu  in the setting of singular varieties, one needs a uniform Sobolev inequality as Conjecture \ref{conj:Sobolev-conj-global}.  
As in the proof of Lemma \ref{lem:Sobolev-prop-single}, thanks to the compactness assumption, we  can reduce it to the following local version.  

\begin{conj}\label{conj-uniform-Sob}
Let $X\subseteq \mathbb{C}^n \times \mathbb{C}^k = N$ be an $m$-dimensional integral complex analytic subvariety defined around the origin $o$.  
Let $(\mathbf{z},\mathbf{w})$ be a coordinates system of $\mathbb{C}^n \times \mathbb{C}^k$. 
We consider the K\"ahler forms  for $0 \leqslant  \delta \leqslant 1$, 
\[\omega_\delta = \frac{\sqrt{-1}}{2} \partial \dbar ( |\mathbf{z}|^2 + \delta^2 \log\, (|\mathbf{z}|^2) +|\mathbf{w}|^2).
\]
Then there is a neighbourhood $U$ of $o$ in $N$, a constant $C$, independent of $\delta$, such that for any positive $\mathcal{C}^1$ function, compactly supported in the smooth locus of $U\cap X \backslash \{o\}$, we have 
\[
\int_{X} h \omega_\delta^m   
\leqslant  C   \int_{X } |\nabla_{X_\delta} h|\omega_\delta^m.
\] 
\end{conj}

Here we denote by $\nabla_{X_\delta}$ the gradient operator on $X$ induced by $\omega_\delta$. 
We note that $\omega_\delta$ can be viewed as a smooth K\"ahler form of the complex manifold  $\widehat{ \mathbb{C}^n} \times \mathbb{C}^k$, where $\widehat{ \mathbb{C}}^n$ is the blow up of $\mathbb{C}^n$ at the origin. 
Furthermore, we observe that there is a constant $A>0$ such that 
\[\omega_0 \leqslant \omega_1 \leqslant \frac{A}{|\mathbf{z}|^2} \omega_0.\]  
We will denote by $\mathbf{B}_r(o)$ and $\mathbf{S}_r(o)$ the closed unit ball and the closed unit sphere in $\mathbb{C}^n$.  
We will prove Conjecture \ref{conj-uniform-Sob} in the following cases.

\begin{prop}\label{prop:certain-uniform-Sobolev} 
Conjecture \ref{conj-uniform-Sob} holds if $k=0$.
\end{prop}

\begin{cor}\label{cor:uniform-sob-product}
Conjecture \ref{conj-uniform-Sob} holds  if $X = \mathbb{C}^n \times \mathbb{C}^k$.
\end{cor} 

\begin{proof}
Let $Y= \mathbb{C}^n$. Then there is a uniform Sobolev constant $C_1$ for 
$$(Y, \frac{\sqrt{-1}}{2} \partial \dbar ( |\mathbf{z}|^2 + \delta^2 \log\, (|\mathbf{z}|^2)) )$$ 
by Proposition \ref{prop:certain-uniform-Sobolev}. 
Hence there is a uniform Sobolev constant $C$ for $(X,\omega_\delta)$, depending on $C_1$ and $k$. 

To see this, we only need to apply the following argument several times. 
Let  $(M,g)$ is a Riemannian manifold of real dimension $d$,  with a Sobolev constant $D$, 
then $M'=M\times \mathbb{R}$ admit Sobolev inequalities with Sobolev constant $D'=D^{\frac{d}{d+1}}$. 
Indeed, for any compactly supported positive $\mathcal{C}^1$ function $h$ on $M'$, we  apply  H\"older's inequality  and get 
\begin{eqnarray*}
\int_{M'} h^{\frac{d+1}{d}} 
&=& \int_{-\infty}^{+\infty} \Big( \int_{M_s} h^{\frac{d+1}{d}}  \Big)  \mathrm{d}s   \\
& \leqslant  &  \int_{-\infty}^{+\infty} \Big( \Big(  \int_{M_s} h \Big)^{\frac{1}{d}}  \Big( \int_{M_s} h^{\frac{d}{d-1}} \Big)^{\frac{d-1}{d}}   \Big)  \mathrm{d}s, 
\end{eqnarray*}
where $M_s=M\times \{s\}$ for any $s\in \mathbb{R}$.  
By assumption, we deduce that 
\begin{eqnarray*}
\int_M h^{\frac{d+1}{d}} 
&\leqslant& D\int_{-\infty}^{+\infty}  \Big( \int_{M_s} h     \Big)^{\frac{1}{d}} \Big( \int_{M_s} |\nabla_{M_s} h|     \Big)  \mathrm{d}s\\
&\leqslant& D\int_{-\infty}^{+\infty}  \Big( \int_{M_s} h     \Big)^{\frac{1}{d}} \Big( \int_{M_s} |\nabla_{M'} h|     \Big)  \mathrm{d}s.
\end{eqnarray*}
We also note that, for any $x\in M$, we have 
\[
h(x,s) 
\leqslant \int_{-\infty}^{+\infty} |\nabla_{M'} h(x,t)| \mathrm{d}t.
\]
Hence 
\[\int_{M_s} h   \leqslant \int_{M'} |\nabla_{M'} h|\] 
for any $s \in \mathbb{R}$.   
In the end, we conclude that 
\begin{eqnarray*}
\int_M h^{\frac{d+1}{d}} 
 \leqslant 
D \Big(\int_{M'} |\nabla_{M'} h| \Big)^{\frac{d+1}{d}}.
\end{eqnarray*}
This completes the proof of the Corollary. 
\end{proof}

The idea for Proposition \ref{prop:certain-uniform-Sobolev} is to decompose the space $\mathbb{C}^n$ into two pieces, depending on $\delta$.  
One of them  is $\mathbf{B}_{\delta\zeta}(o)$, so that we can apply \cite[Theorem 2.1]{HoffmanSpruck1974} inside. 
On its complement, $\omega_\delta$ equivalent to $\omega_0$, with estimates depending only on $\zeta$. 
To carry out this method, we will need the following lemmas.

\begin{lemma}\label{lemma:volume-estimate-blow-up} 
Let $X\subseteq \mathbb{C}^n$ be a  analytic subvariety of pure dimension $m$, defined around the origin $o\in \mathbb{C}^n$. 
Then for $k\in [0, 2m)$, there is a constant $C_k$,   such that for  $\rho >0$, we have the following inequality
\[
\int_{\mathbf{B}_\rho(o) \cap X}  |\mathbf{z}|^{-k}  
\leqslant C_k \rho^{2m-k}.
\] 
\end{lemma}

\begin{proof}
The co-area formula shows that 
\[
\int_{\mathbf{B}_\rho(o) \cap X} |\mathbf{z}|^{-k} = \int_0^\rho |\mathbf{z}|^{-k}   \mathrm{d} \mathrm{Vol}\,((\mathbf{B}_x(r) \cap X)).
\]
There is a constant $D>0$ such that  that 
\[ \mathrm{Vol}\,(\mathbf{B}_r(x) \cap X) \leqslant  D r^{2m}. \] 
Hence,   by Lemma \ref{lemma:integration-by-parts} below, there is a constant $C_k$, depending only on  $k$ and $D$, such that 
\[
\int_{\mathbf{B}_\rho(o) \cap X} |\mathbf{z}|^{-k}     
\leqslant C_k\rho^{2m-k}.
\] 
\end{proof}

\begin{lemma}\label{lemma:integration-by-parts}
Let $\rho_0>0$ and   $F\colon [0,\rho_0] \to \mathbb{R}_{\geqslant 0}$  an increasing function. 
Assume that there is a constant $C$ and an integer $d\geqslant 3$ such that $F(r) \leqslant C r^d$ for all $r\in [0,\rho_0]$. 
Then for $k\in [0, d)$, there is a constant $C_k$ such that for any $\rho \in (0,\rho_0]$, we have
\[
\int_0^\rho r^{-k}\mathrm{d} F(r) \leqslant C_k \rho^{d-k}.
\]
\end{lemma}

\begin{proof}
We fix some $\rho \in (0,\rho_0]$. 
For any  $0< \epsilon \leqslant \rho$, 
the integration by parts for Riemann–Stieltjes integral implies that 
\[ 
\int_\epsilon^\rho  r^{-k} \mathrm{d} F(r)  
= \Big[ r^{-k} F(r) \Big]_\epsilon^\rho 
+ k\int_\epsilon^\rho  r^{-k-1} F(r) \mathrm{d}r.
\]
By assumption, the right-hand-side in the previous equality is  bounded from above by
\[  
C\rho^{d-k} + \frac{kC}{d-k} \rho^{d-k}. 
\]
We can hence deduce the lemma by tending $\epsilon$   to zero.
\end{proof}
  
As a consequence, we deduce the following lemma.
 
\begin{lemma}
\label{lemma:uniform-volume-estimate} 
For any $S >0$, there is some $\zeta>0$ such that, for any $0 < \delta \leqslant 1$,  the following inequality holds,
\[
\int_{X\cap   \mathbf{B}_{\delta\zeta }(o) } 
\omega_\delta^m
 \leqslant S\delta^{2m}.
\] 
\end{lemma} 

\begin{proof}   
We observe that  
\[
\omega_\delta^m|_X \leqslant (\sum_{j=0}^{m-1}  \binom{m}{j}  A^j \delta^{2j} |\mathbf{z}|^{-2j}  )\omega_0^m|_X  
+  \delta^{2m} \omega_1^{m} |_X.
\]

On the one hand, by applying   Lemma \ref{lemma:volume-estimate-blow-up}, there is a constant $D_1$ such that, for any $\rho > 0$,    
\[
\int_{X \cap  \mathbf{B}_{\delta \rho}(o) }     (\sum_{j=0}^{m-1} c_j   \delta^{2j} |\mathbf{z}|^{-2j}  )\omega_0^m
\leqslant D_1 \rho^{2}\delta^{2m}.  
\]
On the other hand, we remark that $(\mathbf{B}_{\delta \rho}(o)\backslash \{o\}, \omega_1)$ is the $\delta\rho$-tubular neighbourhood of $E$ in $(\widehat{\mathrm{C}}^n, \omega_1)$, where $E$ is the exceptional divisor of 
$\widehat{\mathrm{C}}^n \to \mathbb{C}^n$. 
Hence there is a constant $D_2$ such that
\[
\int_{X \cap  \mathbf{B}_{\delta \rho}(o) }   \omega_1^{m}|_X
\leqslant D_2\rho^{2}\delta^{2}.
\]
Therefore, there is a constant $D_3$ such that 
\[ 
\delta^{-2m}  \int_{X\cap   \mathbf{B}_{\delta \rho }(x) } 
\omega_\delta^m   \leqslant D_3 \rho^2.   
\]
This implies the existence of $\zeta$.
\end{proof}

In order to glue the estimates of the two regions, we will require the following results. 

\begin{lemma}\label{lemma:Whitney-neighbourhood}
There is an open neighbourhood $U$ of $o$ in $N$, and  a number $\gamma>0$, such that for any smooth point $x$ of $X\cap U$, different from $o$, we have the following  inequality
\[ 
\langle \mathbf{e}_x, \mathbf{e}_x^t \rangle \geqslant \gamma,
\]
where $\mathbf{e}_x$ is the real unit directional vector pointing from $o$ to $x$,  $\mathbf{e}_x^t$ is its orthogonal projection onto
the real tangent space of $X$ at $x$, and the inner product is taken with respect to the canonical Euclidean product. 
\end{lemma}

\begin{proof}
It follows from the Whitney condition (b) of Whitney stratifications, whose existence follows from \cite[Theorem 19.2]{Whitney1965}. 
\end{proof}

\begin{lemma}\label{lemma:slice-estimate}
With the   open neighbourhood $U$ in Lemma \ref{lemma:Whitney-neighbourhood}, for any $\rho >0$, for any  positive $\mathcal{C}^1$ function $h$, compactly supported in the smooth locus of $U\cap X \backslash \{o\}$,
we have 
\[
\int_{X \cap \mathbf{S}_\rho(o)} h \leqslant \frac{1}{\gamma^2} \int_{X \backslash \mathbf{B}_\rho(o)} |\nabla_X h|.
\] 
Here the Riemannian metric on $X$ is the one induced by the Euclidean metric on $N$.
\end{lemma}

\begin{proof}
Let $r$ be the distance function to the origin $o$ in $\mathbb{C}^n$. 
Then, as shown in \cite[Equation (3.6)]{HoffmanSpruck1974}, outside the origin $o$, we have 
\[|\nabla_X r| \leqslant |\nabla_N r| =1  \mbox{ and }  \mathrm{div}_X (r\nabla_X r) = 2m.\]
As a consequence, we deduce that 
\[  \mathrm{div_X} ( \nabla_X r) 
= \frac{1}{r}(  \mathrm{div}_X ( r\nabla_X r)- |\nabla_X r|^2 ) 
\geqslant 0.
\] 
Since $X$ is a minimal subvariety by \cite[Theorem 3.1.2]{Simons1968}, the divergence formula as in \cite[Equation (3.2)]{HoffmanSpruck1974}   implies that 
\[ 0 \leqslant \int_{X \backslash \mathbf{B}_\rho(o)} h \cdot \mathrm{div}_X ( \nabla_X r) 
= -\int_{X \backslash \mathbf{B}_\rho(o)}   \langle \nabla_X h, \nabla_N r \rangle 
- \int_{X \cap \mathbf{S}_\rho(o)} h  |\nabla_X r|^2.
\]
We note that, at any point $x\in X\backslash \{o\}$,  we have
  $ \nabla_X r (x)  = \mathbf{e}_x^t$. 
Thus by Lemma \ref{lemma:Whitney-neighbourhood}, we deduce that  $|\nabla_X r|^2  \geqslant \gamma^2.$ 
The lemma then follows. 
\end{proof}

Now we can conclude the proof of Proposition \ref{prop:certain-uniform-Sobolev}.

\begin{proof}[{Proof of Proposition \ref{prop:certain-uniform-Sobolev}}] 
Let $U$ be as in Lemma \ref{lemma:Whitney-neighbourhood}. 
By scaling, we see that there are constant $b,R>0$, such that the (real) sectional  curvatures  of $(\widehat{\mathbb{C}^n}, \omega_\delta)$ is bounded from above by $\frac{b^2}{\delta^2}$, and the injectivity radius of  $(\widehat{\mathbb{C}^n}, \omega_\delta)$ is at least $\delta R$. 
Let $S$ be small enough so that the following inequality is well-defined and valid, 
\[ b^{-1}\sin^{-1}( b( 2V_{2m}^{-1} S)^{\frac{1}{2m}})  \leqslant  \frac{ R}{2},\]
where $V_{2m}$ is the volume of Euclidean $2m$-dimensional unit ball. 
By Lemma \ref{lemma:uniform-volume-estimate}, there is some  $\zeta>0$ such that 
for any $0 < \delta \leqslant 1$,  the following inequality holds,
\[
\int_{X\cap   \mathbf{B}_{\delta\zeta }(o) } 
\omega_\delta^m
 \leqslant S\delta^{2m}.
\] 

Then by \cite[Theorem 2.1]{HoffmanSpruck1974}, there is a constant $K>0$ such that for any  positive $\mathcal{C}^1$ function $g$, compactly supported in the smooth locus of $X\cap \mathbf{B}_{\delta\zeta }(o)  \backslash \{o\}$, we have 
\[ 
\Big(\int_{X \cap \mathbf{B}_{\delta\zeta}(o) } g^{\frac{2m}{2m-1}} \omega_\delta^m \Big)^{\frac{2m-1}{2m}}  
\leqslant K 
\int_{X \cap \mathbf{B}_{\delta\zeta}(o) } |\nabla_{X_\delta} g|  \omega_\delta^m.
\]

Let $h$ be a positive $\mathcal{C}^1$ function, compactly supported in the smooth locus of $X\cap U \backslash \{o\}$.  
By multiplying $h$ by smooth cut-off functions, which approximate to the characteristic function of $\mathbf{B}_{\delta\zeta }(o)$, the previous inequality implies that  
\[ 
\Big(\int_{X \cap \mathbf{B}_{\delta\zeta}(o) } h^{\frac{2m}{2m-1}} \omega_\delta^m \Big)^{\frac{2m-1}{2m}}  
\leqslant K 
\int_{X \cap \mathbf{B}_{\delta\zeta}(o) } |\nabla_{X_\delta} h|  \omega_\delta^m
+  K \int_{X_\cap \mathbf{S}_{\delta\zeta}(o)} h  \mathrm{d}A_\delta,
\]
where $\mathrm{d}A_\delta$ is the measure induced by $\omega_\delta$.

We note that $\omega_0 \leqslant \omega_\delta \leqslant (1+\frac{A}{\zeta})^2\omega_0$ over 
$\{ |\mathbf{z}| \geqslant  \delta \zeta \}.$ 
By multiplying $h$ cut-off functions which approximate to the characteristic function of $\mathbb{C}^{n}\backslash \mathbf{B}_{\delta\zeta }(o)$, and then 
by applying \cite[Theorem 2.1]{MichaelSimon1973}, we deduce that there is a constant $K_1$, depending only on $m,A,\zeta$, such that 
\[ 
\Big(\int_{X \backslash \mathbf{B}_{\delta\zeta}(o) } h^{\frac{2m}{2m-1}} \omega_\delta^m \Big)^{\frac{2m-1}{2m}}  
\leqslant K_1 
\int_{X \backslash \mathbf{B}_{\delta\zeta}(o) } |\nabla_{X_\delta} h|  \omega_\delta^m
+  K_1 \int_{X_\cap \mathbf{S}_{\delta\zeta}(o)} h  \mathrm{d}A_\delta,
\]
We also note that 
\[\int_{X_\cap \mathbf{S}_{\delta\zeta}(o)} h  \mathrm{d}A_\delta 
\leqslant  
(1+\frac{A}{\zeta})^{2m-1}\int_{X_\cap \mathbf{S}_{\delta\zeta}(o)} h  \mathrm{d}A_0.
 \]
By Lemma \ref{lemma:slice-estimate},  
\begin{eqnarray*}
\int_{X_\cap \mathbf{S}_{\delta\zeta}(o)} h  \mathrm{d}A_0 
&\leqslant& \frac{1}{\gamma^2}\int_{X \backslash \mathbf{B}_{\delta\zeta}(o) } |\nabla_{X_0} h|  \omega_0^m \\
&\leqslant& \frac{(1+\frac{A}{\zeta})}{\gamma^2} \int_{X \backslash \mathbf{B}_{\delta\zeta}(o) } |\nabla_{X_\delta} h|  \omega_\delta^m. 
\end{eqnarray*} 
Finally we obtain that 
\begin{eqnarray*}
\Big(\int_X  h ^{\frac{2m}{2m-1}} \omega_\delta^m  \Big)^{\frac{2m-1}{2m}}   
&\leqslant&  
(K+K_1) (1+ (1+\frac{ A}{\zeta})^{2m}\gamma^{-2})\int_{X } |\nabla_{X_\delta}h|\omega_\delta^m.
\end{eqnarray*}
This completes the proof of the proposition. 
\end{proof}

\section{Existence of admissible metrics}
\label{section:admissible}
  
In this section, we will assume  Conjecture \ref{conj:Sobolev-conj-global}   and then  prove Theorem \ref{thm:existence-admissible}. 
That is,   the existence of admissible solutions until infinite time of the heat equation (\ref{equa:heat-equation-metric}), 
\begin{equation*}
\begin{cases}
h_t^{-1}(\frac{\mathrm{\partial}}{\mathrm{\partial}t}h_t)  = -(\sqrt{-1}\Lambda_{\omega} F_t - \lambda_\omega(\cF) \mathrm{Id}_\cF)\\
\mathrm{det}\, h_t \equiv 1
\end{cases}
\end{equation*}
with an initial metric $H_0$ constructed as in Section \ref{section:initial-metric}.  
The method goes back to \cite{BandoSiu1994}. 
We will first solve the heat equation on a desingularization $r\colon \widehat{X} \to (X,\omega)$, with respect to small  perturbation of $r^*\omega$. 
Then we will try to converge the solutions and get the solution on $(X,\omega)$. 
In order to guarantee the convergence, we will prove  some uniform estimates on the solutions with respect to perturbed metrics, in Section \ref{subsection:uniform-estimate}. 
Once we get the uniform estimates, we will use induction to conclude  Theorem \ref{thm:existence-admissible} in Section \ref{subsection:solution-heat-solution}. 
In the last subsection, we will deduce some estimate on the shifted solution $g(t)=h(1)^{-1}h(t)$. 
It will serve the proof on existence of Hermitian-Yang-Mills metric in Section \ref{section:HYM-metric}.

\subsection{Uniform estimates}\label{subsection:uniform-estimate}
 
Let $(X,\omega)$ be a compact K\"ahler variety of complex dimension $n$. 
Assume that we have the following Sobolev inequality   
\[ (\int_{X} |f|^{\frac{2n}{ n-1} })^{\frac{n-1}{n}}  \leqslant C_S \int_{X} (||\nabla f||^2 + |f|^2)  \]
for any $\mathcal{C}^1$ function $f$, compactly supported in the smooth locus of $X$. 
Let  $\cF$ be a coherent sheaf on $X$. 
Let $X^\circ$  be any smooth Zariski open dense subset of $X$ over which $\cF$ is locally free.
Let $H_0$ be a Hermitian metric on $\cF $, with Chern curvature $F_0$ such that 
\[   \int_{z\in  X^\circ }  ||\Lambda_{\omega}  F_0(z)||_{H_0} \omega^n \leqslant I, \]
for some constant $I.$
Assume that $h_t$ is a solution of the heat equation (\ref{equa:heat-equation-metric}) for $\cF|_{X^\circ}$ with initial metric $H_0$. 
We also suppose that 
\[ ||\Lambda_{\omega} F (x,t)||_{H_t}   \leqslant \int_{z\in  X^\circ } K  (x,z,\frac{1}{2}t) ||\Lambda_{\omega}  F_0(z)||_{H_0} \omega^n, \]
where $K$ is the  heat kernel of $(X^\circ,\omega)$.

\begin{lemma}\label{lemma:bound-trace-curvature}
We have the following two estimate for  $||\Lambda_{\omega} F (x,t)||_{H_t}$.
\begin{enumerate}
\item  If  $||\Lambda_\omega F_0(x)||_{H_0} \leqslant B$ for some positive constant  $B$, 
then 
\[ ||\Lambda_{\omega} F (x,t)||_{H_t} \leqslant B \] 
for all $(x,t) \in X^\circ \times [0,\infty)$.

\item Let $U \subseteq V$ be two precompcat open subset of $X^\circ$.
Assume that there are strictly positive numbers $S=S(V)$ and $\delta = \delta(U,V)$ such that 
\[ \mathrm{dist}\,(U, X\backslash V) \geqslant \delta,\]
and  that 
\[\sup_{x\in V} ||\Lambda F_0(x)||_{H_0} \leqslant S.\] 
Then  there exists a positive constant $M$, depending on  $\delta,S,n,C_S,I$  such that 
\[ ||\Lambda_{\omega} F (x,t)||_{H_t} \leqslant M  \] 
for all $(x,t) \in U\times (0,\infty)$. 
\end{enumerate}
 
\end{lemma}
\begin{proof}
For (1), we use the fact that $\int_{z\in  X^\circ } K  (x,z,\frac{1}{2}t)  \omega^n \leqslant 1$ to deduce that 
\[ ||\Lambda_{\omega} F (x,t)||_{H_t}  \leqslant \int_{z\in  X^\circ } K  (x,z,\frac{1}{2}t) ||\Lambda_{\omega}  F_0(z)||_{H_0} \omega^n \leqslant B. \]
For (2), we follow the idea of \cite[Lemma 2.2]{LiZhangZhang2017}. 
We fix $(x,t)$  and  split the integral 
$\int_{z\in  X^\circ } K  (x,z,\frac{1}{2}t)||\Lambda_{\omega}  F_0(z)||_{H_0} \omega^n  = I_1 + I_2,$ 
where 
\[I_1 = \int_{z\in  V } K  (x,z,\frac{1}{2}t) ||\Lambda_{\omega}  F_0(z)||_{H_0} \omega^n 
\]
and 
\[ I_2 = \int_{z\in  X\backslash V } K  (x,z,\frac{1}{2}t) ||\Lambda_{\omega}  F_0(z)||_{H_0} \omega^n.
\] 
Then we have 
\[I_1 \leqslant S  \int_{z\in  X^\circ } K  (x,z,\frac{1}{2}t) \omega^n  \leqslant S. \]
For $I_2$, we see that $r(x,z) \geqslant  \delta$ for any $z\in X\backslash V$. 
Thus by Corollary \ref{cor:sobolev-heat-estimate-2points}, for such $z\in X\backslash V$, we have 
\[ K(x,z,\frac{1}{2}t) \leqslant \frac{2^{2n+2}e^{t_0}}{\sigma(\frac{\mu}{2} t)} \exp(-\frac{2\delta^2}{5t}), \]
where $t_0>0$ and $\mu>0$ depends only on $n$ and 
\begin{equation*} 
\sigma(t) = 
    \begin{cases}
       \Big(\frac{1}{ nC_S} \Big)^{n} e^{-t} t^{n} & \mbox{for } 0 < t<t_0\\
        \Big(\frac{1}{ nC_S} \Big)^{n} e^{-t_0} t_0^{n} & \mbox{for } \mbox{for } t \geqslant t_0 \\
    \end{cases}       
\end{equation*} 
In particular, the right-hand-side of the last inequality is a continuous function of $t$, whose coefficients depends only on $\delta$, $n$ and $C_S$. 
Moreover, it is bounded for $t\in (0,\infty)$. 
Hence there is a positive constant $M'$ depending only on $\delta$,  $n$ and $C_S$, such that 
$K(x,z,\frac{1}{2}t) \leqslant M'(\delta, n,C_S)$ for $z \in X\backslash V$. 
Therefore, we have $I_2\leqslant M'I$. 
Let $$M = M'I+S.$$  
Then we get the desired estimate.
\end{proof}

\begin{lemma}\label{lemma:bound-norm-h}
With the notation above,  assume that    $|\lambda_\omega (\cF)| = |\frac{  \mu_\omega(\cF)}{\mathrm{Vol}_\omega X}| \leqslant \lambda$ for some positive number $\lambda$.  
  We have the following two estimate for  $\mathrm{tr}\, h_t$.
\begin{enumerate}

\item  If  $||\Lambda_\omega F_0(x)||_{H_0} \leqslant B$ for some positive number $B$, 
then there is a positive constant $\widehat{C}$ depending on $B,\lambda$, such that 
\[ |\frac{\partial}{\partial t}  \mathrm{tr}\,  h|   \leqslant \widehat{C}\mathrm{tr}\,  h.\] 
As a consequence,   
for all $(x,t) \in X^\circ \times [0, \infty)$, we have
\[
\mathrm{tr}\,  h_t \leqslant   (\mathrm{rank}\,\cF) \exp ( \widehat{C} t).
\]

\item Let $U \subseteq V$ be two precompcat open subset of $X^\circ$.
Assume that there are strictly positive numbers $S=S(V)$ and $\delta = \delta(U,V)$ such that 
\[ \mathrm{dist}\,(U, X\backslash V) \geqslant \delta,\]
and  that 
\[\sup_{x\in V} ||\Lambda F_0(x)||_{H_0} \leqslant S.\] 
Then there is a positive constant $C$ depending on $\delta, S, n, C_S, I,  \lambda$, such that 
\[ |\frac{\partial}{\partial t}  \mathrm{tr}\,  h|   \leqslant C\mathrm{tr}\,  h.\] 
As a consequence,  for all $(x,t) \in U\times [0, \infty )$, we have  
\[\mathrm{tr}\,  h_t \leqslant   (\mathrm{rank}\,\cF) \exp ( C t). \] 
  
\end{enumerate}
\end{lemma}
  
\begin{proof}
From  the heat equation (\ref{equa:heat-equation-metric}), we have 
\begin{eqnarray*}
 \frac{\partial}{\partial t}\mathrm{tr}\, h_t  =     -  \mathrm{tr}\,  \Big(h_t(\sqrt{-1}\Lambda_{\omega} F_t - \lambda_\omega(\cF) \mathrm{Id}_\cF) \Big).  
\end{eqnarray*}
For simplicity we set  $g = \Lambda_{\omega} F_t$. 
Then we have $\mathrm{tr}\, (h_tg)  = \mathrm{tr}\, (gh_t)= \mathrm{tr}\, (gH_t^{-1} H_t h_t)$. 
Since $h_t$ is selfadjoint with respect to $H_t$, it follows that 
\begin{eqnarray*}
|\mathrm{tr}\, (h_tg)| &=& |\mathrm{tr}\, (gH_t^{-1} h_t^*  H_t  )| \\
&\leqslant&  \sqrt{2|\mathrm{tr}\, (gH_t^{-1} g^* H_t)| |\mathrm{tr}\, (h_tH_t^{-1} h_t^* H_t)|}\\
& = &  \sqrt{2} ||g||_{H_t}||h||_{H_t}.
\end{eqnarray*}
We note that $||h_t||_{H_t }   = \sqrt{ \mathrm{tr}\, (h_t^2)}$. 
Moreover, since  all eigenvalues of $h_t$ are positive real numbers, we deduce that  
$ \sqrt{ \mathrm{tr}\, (h_t^2)} \leqslant \mathrm{tr}\, h_t.$
Hence we get 
\begin{eqnarray*}
 \Big|\mathrm{tr}\,  \Big(h_t(\sqrt{-1}\Lambda_{\omega} F_t - \lambda_\omega(\cF) \mathrm{Id}_\cF) \Big) \Big|
 \leqslant   (|| \sqrt{ 2}\Lambda_{\omega} F_t  ||_{H_t} + \lambda) \mathrm{tr}\, h_t.
\end{eqnarray*} 

We will prove both cases of the lemma at the same time. 
By Lemma \ref{lemma:bound-trace-curvature}, there is a constant $C'$, depending on $B,  \lambda $ for case (1), and depending on $ \delta, S, n, C_S, I,  \lambda $ for case (2), such that,  on the corresponding domains,  
\[ |\frac{\partial}{\partial t}  \mathrm{tr}\,  h|   \leqslant C'\mathrm{tr}\,  h.\]
It follows that 
\[\mathrm{tr}\,  h_t \leqslant (\mathrm{tr}\,  h_0 )\exp (C't) = (\mathrm{rank}\,\cF) \exp (C't).
\]
This completes the proof of the lemma.
\end{proof}

We will also need the following lemma.  
  
\begin{lemma}  \label{lemma:uniform-integral-trace}
The integrals 
\[\int_{X}  ||\Lambda_\omega F_t||_{H_t}\]  
are uniformly integrable when $t \to 0$, and converge to
\[ \int_{X}  ||\Lambda_\omega F_0||_{H_0}.\]  
\end{lemma}  
  
\begin{proof}
For any positive number $\eta$, we denote  
\[X_{\eta} = \{ x\in X^\circ  \ | \  \mathrm{dist}(x, X \backslash X^\circ) \geqslant \eta \}.\]
Fix some $\epsilon > 0$, it is enough to show that there is some $\delta > 0$, such that 
\[ \int_{x\in X\backslash X_{\frac{1}{2}\delta}} \int_{y\in X} K(x,y,\frac{1}{2}t) ||\Lambda_\omega F_0(y)||_{H_0}  \leqslant \epsilon.\]
for all $t$ small enough.

We   choose $\delta$ so that 
$\int_{y\in X\backslash X_{\delta}} ||\Lambda_\omega F_0(y)||_{H_0} \leqslant \frac{1}{2}\epsilon$.
Then we decompose 
\[ \int_{x\in X\backslash X_{\frac{1}{2}\delta}} \int_{y\in X} K(x,y,\frac{1}{2}t) ||\Lambda_\omega F_0(y)||_{H_0}  = I_t + J_t,\]
such that 
\[ I_t = \int_{x\in X\backslash X_{\frac{1}{2}\delta}} \int_{y\in X_\delta} K(x,y,\frac{1}{2}t) ||\Lambda_\omega F_0(y)||_{H_0}, \]
and that 
\[ J_t = \int_{x\in X\backslash X_{\frac{1}{2}\delta}} \int_{y\in X\backslash X_\delta} K(x,y,\frac{1}{2}t) ||\Lambda_\omega F_0(y)||_{H_0}.\]
From the Gaussian upper bound of Corollary \ref{cor:sobolev-heat-estimate-2points} on $K(x,y,\frac{1}{2}t)$, there are positive constants $\tau$ and $ C $ small enough, such that $K(x,y,\frac{1}{2}t) \leqslant C$  for $t\leqslant \tau$,  $x\notin X_{\frac{1}{2}\delta}$ and $y\in X_{\delta}$, 
so that 
\begin{eqnarray*}
I_t & \leqslant &  \int_{x\in X\backslash X_{\frac{1}{2}\delta}} \int_{y\in X_\delta} C ||\Lambda_\omega F_0(y)||_{H_0} \\
 &\leqslant & \mathrm{Vol}_\omega X \cdot C  \cdot I\\
 &\leqslant & \frac{1}{2} \epsilon.
\end{eqnarray*}
Furthermore, we have 
\begin{eqnarray*}
J_t &  \leqslant&  \int_{y\in X\backslash X_\delta} \Big( ||\Lambda_\omega F_0(y)||_{H_0} \int_{x\in X} K(x,y,\frac{1}{2}t) \Big) \\
& = & \int_{y\in X\backslash X_\delta}  ||\Lambda_\omega F_0(y)||_{H_0}  \\
&\leqslant & \frac{1}{2}\epsilon.
\end{eqnarray*}
This completes the proof of the lemma. 
\end{proof}

\subsection{Solutions of the heat equations}
\label{subsection:solution-heat-solution}

The  following local estimate is crucial for the proof  Theorem  \ref{thm:existence-admissible}.   
It is a refined version of \cite[Proposition 1]{BandoSiu1994}.  
The proof remains the same, and follows essentially \cite[Section 6]{Hildebrandt1985}. 
The proof employs the Green functions, which admit estimates depending only  $\lambda$ and $\mu$ below, see   \cite[Theorem 1.1]{GruterWidman1982}.

\begin{lemma} 
\label{lemma:Schauder-estimate-matrix-1}
Let  $(X,\omega)$ be a K\"ahler manifold of dimension $n$ and let  $x\in X$.  
We write $B_{R}(x)$ as  the  ball of radius $R$ with center $x$. 
We assume that $B_R(x)$ is a regular ball and contained in a local chart around $x$.
Assume that  the sectional curvature of $(X,\omega)$ on $B_{R}(x)$ is contained in $[-C,C]$ for some constant $C>0$.
Assume that $h$ and $f$  are  smooth functions   on $B_R(x)$ with values in $m\times m$ Hermitian matrices, such that 
\[ \sqrt{-1} \Lambda \dbar( h^{-1}(\partial h)) = f.\]   
We write $\sqrt{-1}\Lambda  (\mathrm{d}z_i \wedge \mathrm{d}\overline{z_j}) =a^{ij}$ and 
assume that $\lambda |\xi|^2 \leqslant a^{ij}(y) \xi_i \overline{\xi_j} \leqslant \mu |\xi|^2$
for any $\xi \in \mathbb{C}^n$ and $y\in B_{R}(x)$.

Then $h$ is $\mathcal{C}^{1,\alpha}$ in a neighbourhood $U_x$ of $x$. Moreover, there is a number  $N$, depending on $R, C, n,m, ||h||_{\infty},  ||f||_{\infty}, \lambda, \mu$, such that 
 $ ||h||_{\mathcal{C}^{1,\alpha}}  \leqslant N$ 
on $U_x$.  
\end{lemma}

As mentioned earlier, the proof of Theorem  \ref{thm:existence-admissible} is by inductively converging solutions with respect to perturbed metrics. 
We will first  prove  a single induction step. 
We know that a desingularization can be obtained by successively blowing up smooth centers. 

Let $p\colon Y \to X$ be a blowup at smooth center, where $(X,\omega)$ is a compact K\"ahler variety of dimension $n$.   
Let $\eta$ be a K\"ahler form on $Y$ and we set 
$$\omega_\epsilon = p^*\omega + \epsilon \eta$$ for $0<\epsilon \leqslant 1.$
We assume Conjecture \ref{conj:Sobolev-conj-global}. 
Then  there is a uniform Sobolev constant $C_S$ for all $\omega_\epsilon.$

\begin{lemma}
\label{lemma:heat-kernel-convergence} 
Assume  Conjecture \ref{conj:Sobolev-conj-global} holds.  
Let $X^\circ \subseteq X$ be a Zariski open dense subset, disjoint from the center of the blowup. 
Let $Y^\circ=p^{-1}(X^\circ) \cong X^\circ$, 
and let  $K_\epsilon $  be the  heat kernel of $(Y^\circ, \omega_\epsilon )$.
Then   we may extract a sequence $\{\epsilon_i\}$ converging to zero,  such that 
$K_{\epsilon_i}$ converge, smoothly over any compact subset of $X^\circ \times X^\circ \times (0,\infty)$,
to the heat kernel $K$ of $(X^\circ, \omega)$.   
\end{lemma}

\begin{proof}
Let $\Delta_\epsilon $ and $\Delta$ be the Laplace-Beltrami operator of $\omega_\epsilon $ and $\omega$ respectively. 
Then  $\Delta_{\epsilon  }$ converges smoothly to $\Delta$ on every compact subset of $X^\circ$. 
We remark that $K_\epsilon $ is a solution of the heat equation on $X^\circ \times X^\circ$, with respect to the operator 
\[\frac{1}{2}(\Delta_{\epsilon } \oplus \Delta_{\epsilon }) -\frac{ \partial}{\partial t}.\]
We also note that, for any fixed $0<t_1<t_2$, the heat kernels  $K_\epsilon $ are uniformly bounded on   $X^\circ \times X^\circ\times [t_1,t_2]$ by  Conjecture \ref{conj:Sobolev-conj-global}  and Corollary  \ref{cor:sobolev-heat-estimate-2points}.  
Hence by applying Lemma \ref{lemma:heat-kernel-local-uniform-bound}, we may extract a sequence $\{\epsilon _i\}$ converging to zero,  such that $K_{\epsilon _i}$ converge, smoothly over any compact subset of $X^\circ \times X^\circ \times (0,\infty)$, to some smooth function $K$.  
We note that $K$ is bounded on $X^\circ \times X^\circ\times [t_1,t_2]$.

It remains to show that $K$ is the heat kernel of $(X^\circ, \omega)$.  
We first show that it is a fundamental solution for the heat equation. 
Let $g$ be a bounded continuous function on $X^\circ$. 
Then the convolution  
\[
G_{\epsilon }(x) = \int_{y\in (Y^\circ, \omega_\epsilon )} g(y)K_\epsilon (x,y,t) \mathrm{d}y
\]
is a solution of the heat equation on $(Y^\circ,\omega_\epsilon )$. 
By the same argument as in the previous paragraph, by replacing $\{\epsilon _i\}$  with a subsequence if necessary, 
we may assume that $G_{\epsilon _i}$ converge  to some function $G$ on $X^\circ\times [0,\infty)$, smoothly over any compact subset of $X^\circ\times [0,\infty)$. 
As a consequence, $G$ is a solution of the heat equation on $(X^\circ,\omega)$, with initial data $g$. 
Hence we can deduce that $K$ is a fundamental solution if   we can prove
\[
\lim_{\epsilon _i \to 0} \int_{y\in (Y^\circ,\omega_\epsilon )} g(y)K_{\epsilon _i}(x,y,t) \mathrm{d}y   = \int_{y\in (X^\circ,\omega)} g(y)K(x,y,t) \mathrm{d}y.
\]
However, by  Conjecture \ref{conj:Sobolev-conj-global}  and Corollary  \ref{cor:sobolev-heat-estimate-2points}, for any fixed $x,t$,  $K_{\epsilon _i}(x,\cdot,t)$ are uniformly bounded on $Y^\circ$. Hence the last equality holds.

We will now show that $K$ is the heat kernel of $(X^\circ,\omega)$. 
We note that 
\[
\int_{X^\circ} K(x,y,t) \mathrm{d}y \leqslant \lim_{\epsilon _i \to 0} \int_{y\in Y^\circ}  K_{\epsilon _i}(x,y,t).
\]
The right-hand-side above is equal to $1$ by item  (4) of  Corollary  \ref{cor:heat-kernel-conservation}. 
Since  the heat kernel of $(X^\circ, \omega)$ also has the conservation property, it  follows from  the minimality of the heat kernel  that $K$ is the heat kernel of $(X^\circ, \omega)$. 
\end{proof}

Let $\cF$ be coherent sheaf on $X$. 
We denote by $X^\circ$ the largest open subset of $X_\cF$ over which $p$ is an isomorphism, and we set  $Y^\circ = p^{-1}(X^\circ)$.  We will make the following assumptions for Lemma \ref{lemma:heat-solution-induction} below.

\begin{assumption}
\label{assumption:integral}
Suppose that $H_0$ is a Hermitian metric on $\cF$,  with Chern curvature $F_0$ such that
\[ 
\int_{Y^{0}} ||\Lambda_{\omega_\epsilon} F_0||_{H_0} \omega_\epsilon^n
\] 
are uniformly bounded and uniformly integrable (see Remark \ref{rem:uniform-integrable}). 
In particular, 
$$|\lambda_{\omega_{\epsilon}} (\cF)| = \Big|\frac{ \mu_{\omega_{\epsilon}}(\cF)}{\mathrm{Vol}_{\omega_{\epsilon}} Y}\Big|  $$ 
are uniformly bounded as well.
\end{assumption}

\begin{assumption}
\label{assumption:bound}
Assume that the heat equation (\ref{equa:heat-equation-metric}) for $p^*\cF$  admits a solution $h_\epsilon$   over $(Y^\circ,\omega_\epsilon) \times (0,\infty)$, with initial metric $H_0$,  
such that  for every $x\in Y^\circ$, we have  
\[ ||\Lambda_{\omega_\epsilon} F_\epsilon (x,t)||_{H_\epsilon}   \leqslant \int_{z\in Y^\circ } K_{\epsilon} (x,z,\frac{1}{2}t) ||\Lambda_{\omega_\epsilon} F_0(z)||_{H_0} \omega_\epsilon^n, \]
where $F_\epsilon$ is  the Chern curvature for the metrics $H_\epsilon =  H_0h_\epsilon$, 
and $K_\epsilon$ is the heat kernel  of $(Y^\circ, \omega_\epsilon)$. 
\end{assumption}

We note that the norm in the left-hand-side above is with respect to the metric $H_{\epsilon}(t)$.  
Then by Lemma \ref{lemma:heat-kernel-convergence}, there exists a sequence $\{\epsilon_i\}$ converging to zero, such that $K_{\epsilon_i}$ converge to the heat kernel $K$ of $(X^\circ, \omega)$, smoothly on any compact subset of $X^\circ \times X^\circ \times (0,\infty)$. 
We have the following lemma on convergence.  
  
\begin{lemma}  
\label{lemma:heat-solution-induction}
With the assumptions above, by replacing  $\{\epsilon_i\}$ with a subsequence if necessary,    $h_{\epsilon_i}$ converge, smoothly on any compact subset of  $X^\circ \times (0,\infty)$, to a solution $h$ of the heat equation (\ref{equa:heat-equation-metric}) over $(X^\circ,\omega)\times (0,\infty)$, with initial metric $H_0$. 
Moreover, we have  
\[ ||\Lambda_{\omega} F (x,t)||_{H}   \leqslant \int_{z\in  X^\circ } K  (x,z,\frac{1}{2}t) ||\Lambda_{\omega}  F_0(z)||_{H_0} \omega^n,\]
where $F$   is  the Chern curvature for the metrics  $H=H_0h$.
\end{lemma}  
  
\begin{proof}
The assumption of  the item (2) of  Lemma \ref{lemma:bound-norm-h} are satisfied. 
Thus we deduce that  $h_\epsilon$ are uniformly bounded on any compact subset of $X^\circ \times [0,\infty)$. 
Hence by Lemma \ref{lemma:Schauder-estimate-matrix-1}, and by local parabolic estimate, by replacing  $\{\epsilon_i\}$ with a subsequence if necessary,   $h_{\epsilon_i}$ converge to a limit $h$, smoothly on any compact subset of $X^\circ \times (0,\infty)$. 
In particular, $h$ is a solution of the heat equation (\ref{equa:heat-equation-metric}) over $(X^\circ,\omega)\times (0,\infty)$, with initial metric $H_0$. 

Since  the   integrals
\[
 \int_{z\in Y^\circ } ||\Lambda_{\omega_\epsilon} F_0(z)||_{H_0} \omega_\epsilon^n
\]
are uniformly bounded and  uniformly integrable,
 we can conclude the estimate on $||\Lambda_{\omega} F (x,t)||_H$ by taking limits.
\end{proof}  
  
The following lemma will guarantee the  integrability   of $ ||F||^2.$
  
\begin{lemma}  
\label{lemma:induction-L2-F}
With the notation above, assume that  for any fixed $t>0$, there is a constant $L$ such that  
\[
\int_Y || F_{\epsilon} ||^2_{H_{\epsilon},\omega_\epsilon } \omega^n_{\epsilon} 
\leqslant L +  \int_Y  ||\Lambda_{\omega_\epsilon} F_{\epsilon} ||^2_{H_{\epsilon}} \omega^n_{\epsilon},
\]
then
\[
\int_X || F ||^2_{H,\omega}   \omega^n 
\leqslant L +  \int_X  ||\Lambda_\omega  F ||^2_{H} \omega^n,
\]
\end{lemma}  

\begin{proof}
For any fixed $t>0$, from the uniform boundedness of the heat kernels $K_\epsilon$ of Corollary \ref{cor:sobolev-heat-estimate-2points}, we obtain that the terms $ ||\Lambda_{\omega_\epsilon} F_{\epsilon} ||^2_{H_{\epsilon,0}}$ are uniformly bounded. 
Hence we can deduce the lemma by taking limits.
\end{proof}

Now we are ready to conclude   Theorem \ref{thm:existence-admissible}.

\begin{proof}[{Proof of Theorem \ref{thm:existence-admissible}}]  
Let $H_0$ be a metric constructed in Lemma \ref{lemma:existence:initial-metric}, and we consider the heat equation (\ref{equa:heat-equation-metric}). We will prove that it admits a solution until infinite time. 

Let $\widehat{X}= X_k  \to \cdots \to X_0 =X$  be  a sequence of blowups, at smooth centers  disjoint from $X_\cF$, such that $r\colon \widehat{X} \to X$ is a desingularization and $r^*(\cF^*)/{(\mathrm{torsion})}$ is locally free. 
In particular, $H_0$ extends to a smooth Hermitian metric $ \widehat{H}_0$ on the locally free sheaf $\widehat{\cF} = (r^*(\cF^*))^*$ by Lemma \ref{lemma:existence:initial-metric}. 

Let $\eta_i$ be a K\"ahler form on $X_i$. By abuse of notation, we still denote by $\eta_i$ its pullback on $\widehat{X}$. 
For $0< \epsilon_i\leqslant 1$, we set 
\[\omega_{\epsilon_1,...,\epsilon_k} = \omega + \epsilon_1 \eta_1 + \cdots + \epsilon_k \eta_k.\]
Then we can solve the heat equation (\ref{equa:heat-equation-metric}) with respect to the compact  K\"ahler manifold  $(\widehat{X}, \omega_{\epsilon_1,...,\epsilon_k})$, 
for the locally free sheaf  $\widehat{\cF}$, with initial metric $\widehat{H}_0$. 
By Lemma \ref{lemma:parabolic-positive-estimate} and  Proposition \ref{prop:max-principal}, we have an estimate for the Chern curvatures of the solutions as follows (we omit the lower index $i$ for simplicity),
\[ ||\Lambda_{\omega_\epsilon} F_\epsilon (x,t)||_{H_\epsilon}   \leqslant \int_{z\in \widehat{X} } K_{\epsilon} (x,z,\frac{1}{2}t) ||\Lambda_{\omega_\epsilon} F_0(z)||_{H_0} \omega_\epsilon^n, \]
where $K_\epsilon$ is the corresponding heat kernel, and $F_0$ is the Chern curvature of $\widehat{H}_0$. 
Moreover, we have 
\[ \int_{\widehat{X}} ||F_\epsilon||^2  \omega_\epsilon^n = \int_{\widehat{X}} ||\Lambda_{\omega_\epsilon} F_\epsilon||^2  \omega_\epsilon^n   + 4\pi^2n(n-1)\Big(2c_2(\widehat{\cF}) -c_1(\widehat{\cF})^2 \Big) \wedge [\omega_\epsilon]^{n-2}.  \]  
We set 
\[L= \sup_{0< \epsilon_1,...,\epsilon_k\leqslant 1} \Big|  4\pi^2n(n-1)    \Big(2c_2(\widehat{\cF}) -c_1(\widehat{\cF})^2 \Big) \wedge [\omega_{\epsilon_1,...,\epsilon_k}]^{n-2} \Big|, \] 
which is a real number since $\widehat{X}$ is a compact manifold. 
Then we have  
\[ \int_{\widehat{X}} ||F_\epsilon||^2  \omega_\epsilon^n \leqslant \int_{\widehat{X}} ||\Lambda_{\omega_\epsilon} F_\epsilon||^2  \omega_\epsilon^n   + L.  \]   

Now by Lemma \ref{lemma:heat-solution-induction}, for any $\epsilon_1,...,\epsilon_{k-1}$ fixed,  we may take a  limit   of a sequence of solutions $h_{\epsilon_1,...,\epsilon_{k}}$ by letting $\epsilon_{k}$ tend to zero. 
Then we get a solution  of the heat equation  on $X_{k-1}$, with respect to the metrics 
\[\omega+\epsilon_1\eta_1 +\cdots + \epsilon_{k-1}\eta_{k-1}.\]
By abuse of notation, we still denote the limit solutions by $h_\epsilon$.   
Then  for solutions $h_\epsilon$ on $X_{k-1}$,  we still have 
\[ ||\Lambda_{\omega_\epsilon} F_\epsilon (x,t)||_{H_\epsilon}   \leqslant \int_{z\in X_{k-1} } K_{\epsilon} (x,z,\frac{1}{2}t) ||\Lambda_{\omega_\epsilon} F_0(z)||_{H_0} \omega_\epsilon^n.\]
Moreover, after Lemma \ref{lemma:induction-L2-F}, we get 
\[ \int_{X_{k-1}} ||F_\epsilon||^2  \omega_\epsilon^n \leqslant \int_{X_{k-1}} ||\Lambda_{\omega_\epsilon} F_\epsilon||^2  \omega_\epsilon^n   + L.\]   

We remark that the Assumption \ref{assumption:integral} is always guaranteed, as explained in Remark \ref{rem:uniform-integrable}. 
Therefore, we can continue carrying out the induction steps of Lemma \ref{lemma:heat-solution-induction}, and finally obtain a solution $h$ on $(X,\omega)$. 
Furthermore, Lemma \ref{lemma:induction-L2-F} shows that, for any $t>0$, the metric $Hh_t$ is admissible.
\end{proof}

\subsection{Shifted solutions of the heat equations}
In Simpson's method in \cite{Simpson1988}, the initial metric for the heat equation is  assumed to have bounded $\Lambda F$. 
This might not be the case for a metric constructed in Lemma \ref{lemma:existence:initial-metric}, which we used in Theorem \ref{thm:existence-admissible}. 
Thus, in order to prove the existence of admissible  Hermitian-Yang-Mills metric,   
we will consider following shifted solution 
$$g(x,t) = h(x,1) ^{-1} h (x,t+1)  \mbox{ and }G_0  = H(x,1)=H_0h_1 $$
on $(X,\omega)$.
In another word, $g$ is the solution of the heat equation (\ref{equa:heat-equation-metric}) with initial metric $G_0$. Now $G_0$ is  admissible by Theorem \ref{thm:existence-admissible}. 
Let $G=G_0g$ and let $F_G$ be the corresponding Chern curvature. 
We have the following  estimates on $g$.  

\begin{prop}\label{prop:estimate-g}
With the notation above, the following properties hold.
\begin{enumerate}
\item The metrics $G(t)$ are admissible for all $t\geqslant 0$, with uniform upper bounds on $||\Lambda F_{G(t)}||_{G(t)}$ and on $\int_X||F_{G_t}||_{G(t)}^2$.
\item The norm  $||g||_{G_0} = ||g||_{G}$ is  bounded over $X^\circ \times [0,T]$ for any $T \geqslant 0$.
\item For   $t \geqslant 0$, we have
\[ 
\int_{X}   ||(\dbar g(t)) g(t)^{-\frac{1}{2}} ||_{{G_0,\omega}}^2  \omega^n  \leqslant \int_{X}    |\mathrm{tr}\, (g  \Lambda F_{G} - g \Lambda F_{G_{0}}) |    \omega^n.
 \]
\item   $\Delta_{G_{0}} g(t)$ is integrable for  any fixed $t\geqslant 0$.
\end{enumerate}
\end{prop}

\begin{proof}
Since $G(t) = H(t+1)$, it is admissible from Theorem \ref{thm:existence-admissible}.  
The uniform bounds for item (1) can be deduced from Lemma \ref{lemma:bound-trace-curvature} and  Lemma \ref{lemma:induction-L2-F}.
For item (2),  since $G_0$ is admissible,  by  item (1) of Lemma \ref{lemma:bound-norm-h}, 
we see that $\mathrm{tr}\, g$ is  bounded over $X^\circ \times [0,T]$ for any $T>0$. 
Since  $g$ is selfadjoint and definite positive with respect to $G_0$, this implies that $||g||_{G_0}=||g||_{G}$ is bounded. 
We observe that, after Lemma \ref{lemma:simpson-computation}, item (4) follows from item (1) and (3).

Hence it remains to prove item (3). First, we notice that, if $X$ is smooth and $\cF$ is locally free, then  by taking the trace of item (2) of Lemma \ref{lemma:simpson-computation}, we can deduce that
\[ 
\int_{X}   ||(\dbar g) g^{-\frac{1}{2}} ||_{{G_0,\omega}}^2  \omega^n \leqslant 
 \int_{X}     |  \mathrm{tr}\, (g  \Lambda F_{G}  - g \Lambda F_{G_{0}})|    \omega^n,
 \]
as in this case, the integral of $$\mathrm{tr}\, (\Delta_{G_0} g)=\Delta (\mathrm{tr}\, g)$$ is zero. 
For the general case, since the solution $g$ is obtained by converging successively the solutions  on  blowups of $X$, by induction, we only need to prove the following lemma.  
\end{proof}

\begin{lemma}
\label{lemma:induction-L_1^2-g}
With the notation of  Lemma \ref{lemma:heat-solution-induction}, 
we   set 
\[    g_\epsilon (x, t) =  h_\epsilon(x,1)^{-1}h_{\epsilon}(x,t+1), \  G_{\epsilon,0} = H_\epsilon(1) \  \mathrm{and}\  G_\epsilon = G_{\epsilon,0}g_\epsilon. \]
Assume that 
\[ 
\int_{Y^\circ}  || (\dbar g_\epsilon) g_\epsilon^{-\frac{1}{2}}||_{G_{\epsilon,0},\omega_\epsilon}^2  \omega_{\epsilon}^n \leqslant 
  \int_{Y^\circ}     |\mathrm{tr}\, (g_\epsilon \Lambda_{\omega_\epsilon} F_{G_\epsilon} - g_\epsilon \Lambda_{\omega_\epsilon} F_{G_{\epsilon,0}})|    \omega_{\epsilon}^n.
 \]
Then for the shifted solution $g(t) = h(1)^{-1}h(t+1)$, we have 
\[ 
\int_{X^\circ}   ||(\dbar g) g^{-\frac{1}{2}} ||_{{G_0,\omega}}^2  \omega^n \leqslant 
 \int_{X^\circ}     | \mathrm{tr}\, (g  \Lambda_{\omega} F_{G} - g \Lambda_{\omega} F_{G_{0}})|  \omega^n.
 \]
\end{lemma}

\begin{proof}
We note that   $g_{\epsilon_{i}}$ converge to  $g$ smoothly on any compact subset of $X^\circ\times (0,\infty)$. 
Thus it is sufficient to show that 
\[\int_{Y^\circ}     |\mathrm{tr}\, (g_\epsilon \Lambda_{\omega_\epsilon} F_{G_\epsilon} - g_\epsilon \Lambda_{\omega_\epsilon} F_{G_{\epsilon,0}})|    \omega_{\epsilon}^n.
\]
are  uniformly integrable.  
We fix some $t\geqslant 0$.

By Proposition \ref{prop:sobolev-heat-estimate}, the heat kernels $K_\epsilon(\cdot,\cdot, \frac{1}{2})$ are uniformly bounded. 
It follows by assumption that
\[
||\Lambda_{\omega_\epsilon} F_{G_{\epsilon,0}}(z)||_{G_{\epsilon,0}} =  ||\Lambda_{\omega_\epsilon} F_{H_\epsilon }(z,1)||_{H_{\epsilon}(1)} 
\leqslant  \int_{z\in Y^\circ } K_{\epsilon} (x,z,\frac{1}{2}) ||\Lambda_{\omega_\epsilon} F_0(z)||_{H_0} \omega_\epsilon^n
\] 
are uniformly bounded, independent of $\epsilon$. 
Hence by item (1) of Lemma \ref{lemma:bound-trace-curvature},  $||  \Lambda_{\omega_\epsilon} F_{G_\epsilon} ||_{G_{\epsilon}}$  are uniformly bounded. 
By item (1) of   Lemma \ref{lemma:bound-norm-h},  
 $ || g_\epsilon ||_{G_{\epsilon}} = || g_\epsilon ||_{G_{\epsilon,0}}$   are uniformly bounded. 
Hence we get the desire  uniform integrability of the previous paragraph.
This completes the proof of the lemma. 
\end{proof}

The following lemma says that if $\cF$ is stable, then the stability persists during the heat flow.

\begin{lemma}\label{lemma:stability-persist}
Assume that $\cF$ is stable with respect to $\omega$.  
Let $\cE \subseteq \cF$ be a proper non zero  saturated coherent subsheaf, and let $\pi$ be the $G_0$-orthogonal projection  from $\cF$ to $\cE$. 
Assume that  $ || \dbar \pi ||_{G_0}^2 $ is integrable. 
Then 
 \[ \frac{\deg_\omega(\cE, G_0|_\cE)}{\mathrm{rank}\, \cE}  < \frac{\deg_\omega(\cF, G_0)}{\mathrm{rank}\, \cF}. \]
\end{lemma}

\begin{proof} 
We write  $G_{t} = H_{t+1}$ for $t\geqslant -1$.
In particular, the initial metric  $H_0$ is just $G_{-1}.$ 
We first note that $\deg_\omega(\cF, G_t)$ is constant since $\det\, g_t \equiv 1$ for all $t$. 
In the following argument, we will let $t$ tend to $-1$.  

For any $t\geqslant -1$, we denote by $\pi_t$ the $G_t$-orthogonal projection from $\cF|_{X_\cF}$ to $\cE$. 
Then from the Chern-Weil formula, 
\begin{equation}\label{equa-degree}
  \deg_\omega(\cE, G_t|{_\cE}) = \sqrt{-1}\int_X \mathrm{tr}\, \pi_t \Lambda F_{G_t} - \int_X ||\dbar \pi_t||_{G_t}^2.
\end{equation}

Let $p_t = \pi g_t \pi$.  
Then by taking local $G_0$-orthonormal basis, which diagonalizes $\pi$,  we see that 
\[
G_t|_\cE = (G_0 g_t \pi)|_\cE     = G_0|_{\cE} p_t|_{\cE}.
\]
The corresponding Chern curvatures satisfy
\[ F_{G_t|_\cE} = F_{G_0|_\cE}  + \dbar \Big( (p_t|_\cE)^{-1} (\partial_{G_0|_\cE} p_t|_{\cE})\Big).\]
By taking the trace, we get
\[ \deg_\omega(\cE, G_t|_\cE) =\deg_\omega(\cE, G_0|_\cE) +  n\sqrt{-1}\int_X \dbar \mathrm{tr}\, \Big(    (p_t|_\cE)^{-1} (\partial_{G_0|_\cE} p_t|_{\cE}) \Big) \omega^{n-1}. \]
For any $t>-1$, we see that $||\dbar (p_t|_\cE)||_{G_0|_\cE}^2 = ||\dbar p_t||_{G_0}^2$, which is integrable, since $\pi$ and $g_t$ are bounded and $L_2^1$, as discussed in Proposition \ref{prop:estimate-g}. 
Furthermore  $||(p_t|_{\cE})^{-1}||_{G_0|_\cE} $ is bounded, as $g_t$ is bounded with determinant equal to $1$. 
Thus  
\[\mathrm{tr}\, \Big(    (p_t|_\cE)^{-1} (\partial_{G_0|_\cE}  p_t|_{\cE} ) \Big)\]
is square integrable. 
By Corollary \ref{cor:integration-by-part-2}, we deduce that, for $t>-1$,
\[ \deg_\omega(\cE, G_t|_\cE) =\deg_\omega(\cE, G_0|_\cE).\]

Now we assume by contradiction that 
 \[ \frac{\deg_\omega(\cE, G_0|_\cE)}{\mathrm{rank}\, \cE}  \geqslant \frac{\deg_\omega(\cF, G_0)}{\mathrm{rank}\, \cF}= \frac{\deg_\omega(\cF, H_0)}{\mathrm{rank}\, \cF}. \]
Then from the discussion above, for all $t>-1$,
\[ \frac{\deg_\omega(\cE, G_t|_\cE)}{\mathrm{rank}\, \cE}  \geqslant \frac{\deg_\omega(\cF, H_0)}{\mathrm{rank}\, \cF}.\] 
We   let $t$ tend to $-1$ in the equation (\ref{equa-degree}) above.  
We note that $||\pi_t||^2_{G_t}$ is always equal to the rank of $\cE$.
By Lemma \ref{lemma:uniform-integral-trace}, and by Fatou's lemma, we deduce that 
\[  \deg_\omega(\cE, G_{-1}|{_\cE})  \geqslant \lim_{t\to -1}  \deg_\omega(\cE, G_t|{_\cE}). \]
It follows that 
\[\frac{\deg_\omega(\cE, H_0|_\cE)}{\mathrm{rank}\, \cE}  \geqslant \frac{\deg_\omega(\cF, H_0)}{\mathrm{rank}\, \cF}.\]
This contradicts the stability condition.
\end{proof}

\section{Existence of admissible Hermitian-Yang-Mills metrics}
\label{section:HYM-metric}

In this section, we will assume  Conjecture \ref{conj:Sobolev-conj-global} and  prove the existence of   admissible Hermitian-Yang-Mills metrics for stable reflexive
sheaves on a compact normal K\"ahler variety.
We will follow the method of \cite{Simpson1988}. 
  
We first recall the following Donaldson's functional $M(\cdot,\cdot)$, as defined in \cite[Section 4]{Simpson1988}. 
Let   $(\cF, H)$ be  a  Hermitian    reflexive sheaf on a compact normal K\"ahler variety $(X,\omega)$. 
Let $s$ be a  smooth section of $\cE nd(\cF)$, selfadjoint with respect to $H$. 
Locally on $X_\cF$, let $(\mathbf{e}_\alpha)$ be an $H$-orthonormal basis of $\cF$, which diagonalizes $s$. 
In particular, we have $s(\mathbf{e}_\alpha) = \lambda_\alpha \mathbf{e}_\alpha$ for real-valued functions $\lambda_\alpha$. 
Let $(\mathbf{e}^*_\alpha)$ the dual basis of $\cF^*$. 
Let 
\[\Psi(a,b) = \frac{e^{a-b}-1 - (a - b)}{(a-b)^2}
\] 
be a smooth positive function defined on $\mathbb{R}^2$. 
For any section $A$  of $\cE nd(\cF)|_{X_\cF}$ with local expression 
\[A=A_{\alpha,\beta}\mathbf{e}^*_\alpha\otimes \mathbf{e}_\beta,\] 
we define 
\[ \Psi(s)(A) = \Psi(\lambda_\alpha,\lambda_\beta) A_{\alpha,\beta} \mathbf{e}^*_\alpha\otimes \mathbf{e}_\beta. \]
The following lemma follows directly from the definition of $\Psi$.

\begin{lemma}\label{lemma:bounded-psiM2}
With the notation above, assume that $||s||_H   \leqslant C$ for some constant $C$, then there exists some constant $C'>0$, depending only on $C$, such that for any smooth $(0,1)$-form $A$ with values in $\cE nd (\cF)$, we have
\[0 \leqslant -\sqrt{-1}\Lambda  <\Psi(s)A, A>_H  \leqslant -  C' \sqrt{-1}\Lambda  <A, A>_H  = C' ||A||^2_{H,\omega}.\]
\end{lemma}

Now we assume  that  $\Lambda F_H$ is integrable, where $F_H$ is the   Chern curvature of $H$. 
Let  $s$ be a  bounded $L_1^2$ smooth section  of $\cE nd (\cF)$.
Then the Donaldson's functional  $M(H,He^s)$ is defined    as 
\[ M(H,He^s) = \int_X \mathrm{tr}\, (s\sqrt{-1}\Lambda F_H) + \int_X  -\sqrt{-1}\Lambda  <\Psi(s)(\dbar s), \dbar s>_H.\]
We will write $M(H,He^s)= \int_X N_1(H,He^s) + \int_X N_2(H,He^s)$ according to this expression.

\begin{lemma}
\label{lemma:Dfunctional-convergence}
With the assumption above,  assume that $s_1$ and $s_2$ are bounded $L_1^2$ smooth section of $\cE nd (\cF)$, with $\Lambda F_{He^{s_1}}$ integrable. 
Let $\{\varphi_i\}$ be a sequence of cut-off functions as in Lemma \ref{lemma:cut-off}, and we set $u_{1,i} = \varphi_i s_1$ and $u_{2,j} = \varphi_j s_2$. 
Then  we have the following statements.
\begin{enumerate}
\item If $s_2$ has compact support, then  \[ \lim_{i\to \infty}  M(He^{u_{1,i}}, He^{u_{1,i}}e^{s_2})   = M(He^{s_1},He^{s_1}e^{s_2}).
\]
\item \[
\lim_{i\to \infty} M(H ,H e^{u_{1,i}} )   = M(H ,He^{s_1}  ).
\]
\item  
\[ \lim_{i\to \infty}  M(H, He^{u_{1,i}}e^{s_2})   = M(H,He^{s_1}e^{s_2}).
\]
\item \[
\lim_{j\to \infty} M(H ,He^{s_1}e^{u_{2,j}} )   = M(H ,He^{s_1}e^{s_2} ).
\]
\item 
\[ \lim_{j\to \infty} \lim_{i\to \infty}  M(He^{u_{1,i}},He^{u_{1,i}}e^{u_{2,j}})   = M(He^{s_1},He^{s_1}e^{s_2}).
\]
\end{enumerate}
\end{lemma}

\begin{proof}
We will only give detailed proofs for item (1) and item  (3). 
The proof for item (4) is analogue to  the one for item (3). 
Item (2) follows from item (4) by taking $s_2=0$.
Item (5) follows by applying item (1) and then item (2).

For item (1), since $s_2$ has compact support,  there is some $i_0$ such that, if $i\geqslant i_0$, then $s_1=u_{1,i}$  on the support of $s_2$. 
Hence (2) holds.
 
In the remainder of the proof, we focus on item (3). 
Let  $ T_i  = \{\varphi_i \neq 1 \}$. Then the volumes of $T_i$ tend to zero when $i \to \infty$.  We use the notation that $s_1=u_{1,\infty}$ and $\varphi_{\infty} \equiv 1$.
We write $e^{v_i} = h_i = e^{u_1,i}e^{s_2}$ for $i\in \mathbb{N}\cup \{ \infty\}$. 
Then $||v_i||_{\infty}$ are  bounded, and $v_i$ converge to $v_\infty$, smoothly on any compact subset of $X_\cF$.  
We pick a constant $C$ such that  $||v_i||_{\infty},  ||s_1||_{\infty}, ||s_2||_{\infty}, || \dbar s_1||_{2}$ and $|| \dbar s_2||_{2}$ are bounded by $C$. 

First, since $\Lambda F_H$ is integrable,  and since $||v_i||_{\infty} \leqslant C$ for all $i$,  we get
\[
\lim_{i\to \infty}\int_X N_1(H,He^{v_i}) = \int_X N_1(H,He^{v_{\infty}}).
\]
For the other integral in $M(\cdot,\cdot)$, we observe that  
\begin{eqnarray*}
 \int_X N_2(H, He^{v_\infty})  - \int_X N_2 (H, He^{v_i})  
= \int_{T_i}  N_2(H, He^{v_\infty})   - \int_{T_i} N_2 (H, He^{v_i})
\end{eqnarray*}  
By Lemma  \ref{lemma:bounded-psiM2}, we see that there is a  constant  $C_1$, depending only on $C$,   such that 
$N_2(H,He^{v_i}) \leqslant C_1 || \dbar v_i||_{H}^2$
for $i\in \mathbb{N}\cup \{\infty\}$.

We have the following computation, for $i\in \mathbb{N}\cup \{\infty\}$,
\[
\dbar h_i = \dbar(e^{u_{1,i}} e^{s_2}) = \dbar e^{u_{1,i}} e^{s_2} + e^{u_{1,i}} \dbar e^{s_2}.
\]
By Lemma \ref{lemma:h-s-estimate}, there is a constant $C_2$, depending only on $C$, such that 
\[||\dbar e^{u_{1,i}}||_{{H,\omega}}^2 \leqslant C_2 ||\dbar u_{1,i}||_{{H,\omega}}^2 \mbox{ and } 
||\dbar e^{s_2}||_{H,\omega}^2 \leqslant C_2 || \dbar s_2||^2_{H,\omega}.\] 
We note that 
\[||\dbar u_{1,i}||_{H,\omega}^2 = ||\dbar \varphi_i s_1 + \varphi_i \dbar s_1||_{H,\omega}^2 \leqslant 2||\nabla \varphi_i||^2 ||s_1||_{H,\omega}^2  + 2\varphi_i^2 || \dbar s_1||_{H,\omega}^2.\]
Thus there is a constant $C_3$, depending only on $C$, such that 
$$||\dbar h_i||_{H,\omega}^2 \leqslant C_3(|| \dbar s_1||^2_{H,\omega} + || \dbar s_2||^2_{H,\omega} +||\nabla \varphi_i||_\omega^2).$$
By  Lemma \ref{lemma:h-s-estimate} again, this implies that there is a constant $C_4$, depending only on $C$, such that $|| \dbar v_i||_{{H,\omega}}^2 \leqslant C_4(|| \dbar s_1||^2_{H,\omega} + || \dbar s_2||^2_{H,\omega} +||\nabla \varphi_i||_\omega^2)$. 
In conclusion, \[ N_2(H,He^{v_i}) \leqslant C_1 C_4 (|| \dbar s_1||^2_{H,\omega} + || \dbar s_2||^2_{H,\omega} +||\nabla \varphi_i||_\omega^2). \]
Since $||\nabla \varphi_i||_{2} \to 0$ when $i\to \infty$, and since $s_1,s_2$ are $L_1^2$,  we obtain that 
\[ \lim_{i\to \infty} \Big(\int_{T_i}  N_2(H, He^{v_\infty})   - \int_{T_i} N_2 (H, He^{v_i}) \Big) =  0.\]
This completes the proof of the lemma.
\end{proof}

\begin{lemma}
\label{lemma:Dfunctional-addition}
We have 
$$M(H,He^{s_1})  + M(He^{s_1},He^{s_1}e^{s_2})   = M(H,He^{s_1}e^{s_2}).$$
\end{lemma}

\begin{proof}
We first assume that $\Delta_H s_1$ and $\Delta_H s_2$ are integrable. 
Then we can apply the argument of \cite[Proposition 5.1]{Simpson1988}, by just replacing the use of \cite[Lemma 5.2]{Simpson1988} by Corollary \ref{cor:integration-by-part-2}. 
For the general case, we can apply  Lemma \ref{lemma:Dfunctional-convergence}.
\end{proof}

\begin{remark}
We note that, for the proof of the previous lemma,   in \cite[Proposition 5.1]{Simpson1988}, the  integrability of $\Delta_H s_1$  is    mandatory. However, for the solution $h=e^s$ of our heat equation, we can only prove the integrability for $\Delta_H h$. It is not obvious to us that this implies the integrability of $\Delta_H s$. 
For this reason, we prove  Lemma \ref{lemma:Dfunctional-convergence}.
\end{remark}

We replace   \cite[Assumption 3]{Simpson1988}   by the following property on  normal compact K\"ahler varieties.

\begin{lemma}  
\label{lemma:simpson-assumption-3}
Let $(X,\omega)$ be a normal compact K\"ahler variety of complex dimension $n$, and $B$ a positive number.   
Then there exist  constants $C_1$ and $C_2$  such that the following property holds. 
Let $f$ be a smooth bounded   positive  function,  defined over some smooth Zariski open dense  subset of $X$, such that 
\[\Delta f \geqslant -B,\]
and that 
\[ \int_X ||\nabla f||^2 < \infty, \   \int_X |\Delta f|  < \infty.  \] 
Then we have 
\[ ||f||_\infty \leqslant C_1 ||f||_1 + C_2.\]
\end{lemma}

\begin{proof}
We follow  the idea of \cite[Lemma 5.2]{LiZhangZhang2017}. 
By  Lemma \ref{lem:Sobolev-prop-single}, we have a Sobolev inequality  as follows. 
There is a positive constant $C_S$ such that for any   $\mathcal{C}^1$  function $g\geqslant 0$,  compactly supported in the smooth locus of $X$,  
\[ (\int_{X} |g|^{\frac{2n}{n-1} })^{\frac{ n-1}{ n}}  \leqslant C_S \int_{X} (||\nabla g||^2 + |g|^2).\]

We claim that the same inequality holds for any bounded   smooth $L_1^2$ function  $g\geqslant 0$,   defined over some smooth Zariski  open dense  subset  $X^\circ$ of $X$.
Indeed,   we can take a sequence of cut-off functions $\{\varphi_i \}$ on $X^\circ$ as in Lemma \ref{lemma:cut-off}, such that $|| \nabla \varphi_i ||_2 \to 0$. 
By considering the previous inequality for $\varphi_i g$, and by letting $i$ to $\infty$, we deduce the claim.

On the one hand, by taking $g= (f+1)^{\frac{q+1}{2}}$, where $q\geqslant 1$, we obtain that 
\[  (\int_{X} |f+1|^{\frac{n(q+1)}{n-1} })^{\frac{ n-1}{ n}}  \leqslant C_S \int_{X} \Big( \frac{(q+1)^2}{4}|f+1|^{q-1}||\nabla f||^2 + |f+1|^{q+1} \Big). \] 
On the other and, by assumption, we have   
\[ (f+1)^{q} \Delta f \geqslant -B(f+1)^{q}. \] 
By integrating both sides,  we get 
\[  \int_X  (f+1)^{q} \Delta f       \leqslant B \int_X (f+1)^{q}. \]
By the assumption of integrability, we can apply  Corollary \ref{cor:integration-by-part} and obtain that 
\[ \int_X  (f+1)^{q} \Delta f      = \int_X  \nabla ((f+1)^{q})  \cdot  (\nabla f)        = \int_X    q(f+1)^{q-1} ||\nabla f||^2.      \]
Since $f$ is positive, we get $\int_X (f+1)^{q}    \leqslant \int_X (f+1)^{q+1}. $
Therefore we have
\[  (\int_{X} |f+1|^{\frac{n(q+1)}{n-1} }  )^{\frac{ n-1}{ n}}  
\leqslant  
   C_S \Big( \frac{(q+1)^2B}{4q} +1 \Big) \int_X |f+1|^{q+1}. \] 
By taking the $(q+1)$-root, we obtain that 
\[ ||f+1||_{\frac{n(q+1)}{n-1}} \leqslant  \Big( \frac{(q+1)^2BC_S}{4q} +C_S \Big)^{\frac{1}{q+1}} ||f+1||_{q+1}.\]  
By using Moser's iteration procedure, we can conclude the lemma.
\end{proof}

The following statement is a variant of \cite[Proposition 5.3]{Simpson1988}.

\begin{lemma}
\label{lemma:simpson-estimate}
Let $\cF$ be a  reflexive sheaf on a normal compact K\"ahler variety $(X,\omega)$.
We fix a positive number $B$.
Assume that  $H$ is a  Hermitian metric on $\cF$ with Chern curvature $F_H$ such that 
\[ || \Lambda F_H ||_{H} \leqslant B. \]
Assume furthermore that,  for any proper non zero saturated coherent subsheaf $\cE \subseteq \cF $ with  orthogonal projection $\pi$, which is  $L_1^2$, we have 
\[\frac{\deg_\omega(\cE, H|_{\cE})}{\mathrm{rank}\, \cE}  < \frac{\deg_\omega(\cF, H)}{\mathrm{rank}\, \cF}.  \]

Then there exist two positive constant $D_1$, $D_2$ such that 
\[  ||s||_H \leqslant D_1 + D_2 M(H,He^s), \]
for any    smooth section $s$ of $\cE nd({\cF})|_{X_\cF}$ which satisfies the following properties,  
\begin{enumerate} 
\item  $s$ is selfadjoint with respect to $H$,
\item $\mathrm{tr}\, s =0$,
\item $||s||_H$ is bounded,
\item  $ s$ is $L_1^2$,
\item  $ |\Delta  \log \, (\mathrm{tr}\, e^s) | $ is integrable,
\item $|| \Lambda F_s ||_{He^s} \leqslant B$, where $F_s$ is the Chern curvature of the metric $ He^s$.
\end{enumerate}

\end{lemma}
  
\begin{proof}
Let $h=e^s$. By Lemma \ref{lemma:simpson-computation}, we deduce that  
\[
\Delta \log \, (\mathrm{tr}\, h ) \geqslant -2 ( || \Lambda F_H ||_H + ||\Lambda F_{s}||_{Hh}) \geqslant -4B. 
\]
Moreover, since $s$ is bounded and $L_1^2$, by Lemma \ref{lemma:h-s-estimate}, 
we see that $||  \nabla \log \, (\mathrm{tr}\, h ) ||^2$ is integrable.
Hence  by Lemma \ref{lemma:simpson-assumption-3}, there are positive constants $C_1$, $C_2$ such that 
\[ ||\log \, (\mathrm{tr}\, h )||_\infty \leqslant C_1 || \log \, (\mathrm{tr}\,h )||_1 + C_2.\] 
By taking a local $H$-orthonormal basis of $\cF$ which diagonalizes $s$, we see that  there are positive constants $A_1,A_2,A_3,A_4$, depending only on the rank of $\cF$, such that 
\[ A_1 ||s||_H - A_2  \leqslant  \log\, (\mathrm{tr}\, h)   \leqslant  A_3 ||s||_H + A_4.\]
Hence there are  positive constants $C'_1$, $C'_2$ such that 
\[ \sup_X ||s||_H  \leqslant C'_1   \int_X ||s||_H \omega^n + C'_2.\]

Now we can argue as in \cite[Proposition 5.3]{Simpson1988}. 
Assume by contradiction that lemma does not hold. 
Then using the previous inequality, we can show that there is a $L_1^2$-subbundle $\pi \in L^2_1(\cF,\cF)$  on $X_\cF$.  
More precisely, $\pi$ is selfadjoint with respect to $H$,  
$\pi^2 = \pi$, $(\mathrm{Id}_\cF-\pi)\dbar \pi = 0$, 
and  
$$ ||\dbar \pi ||_{\omega,H}^2= -\sqrt{-1}\Lambda \mathrm{tr}\,  (\dbar \pi \partial_H \pi) $$
is   integrable on $X$.  
As in \cite[Section 7]{UhlenbeckYau1986} (or \cite[Theorem 0.1.1]{Popovici2005} for an alternative reasoning), $\pi$ defines a  saturated coherent analytic subsheaf $\cE \subseteq \cF|_{X_\cF}$. 
Moreover, by construction, we have  $\cE \subsetneq \cF|_{X_\cF}$, and
\[ \frac{\deg_\omega(\cE, H|_{\cE})}{\mathrm{rank}\, \cE}  \geqslant \frac{\deg_\omega(\cF, H)}{\mathrm{rank}\, \cF}. \]  
By Lemma \ref{lemma:extension-subsheaf}, 
$\cE$ extends to a coherent subsheaf of $\cF$, globally on $X$. 
We thus obtain a contradiction.
\end{proof}

Now we are ready to prove Theorem \ref{thm:existence-HE}.

\begin{proof}[{Proof of Theorem \ref{thm:existence-HE}}]
Let $H_0$ be a metric constructed as in Section \ref{section:initial-metric}. 
Then by Theorem \ref{thm:existence-admissible}, the heat equation (\ref{equa:heat-equation-metric}) admits a solution with initial metric $H_0$. 
We consider the shifted solution $G=G_t=G_0g_t$ as in Proposition \ref{prop:estimate-g}. We recall the following properties.
\begin{enumerate}
\item The metrics $G_t$ are  admissible for all $t\geqslant 0$, with uniform upper bounds on $||\Lambda F_{G_t}||_{G_t}$ and on $\int_X||F_{G_t}||_{G_t}^2$.
\item The norm  $||g||_{G_0}$ is  bounded over $X^\circ \times [0,T]$ for any $T \geqslant 0$.
\item For   $t \geqslant 0$, we have
\[ 
\int_{X }   ||(\dbar g(t)) g(t)^{-\frac{1}{2}} ||_{{G_0,\omega}}^2  \omega^n \leqslant  
 \int_{X}    |\mathrm{tr}\, (g  \Lambda F_{G} - g \Lambda F_{G_{0}}) |    \omega^n. \]
\item   $\Delta_{G_{0}} g_t$ is integrable for  any fixed $t\geqslant 0$.
\end{enumerate}
We set $g_t=e^{s_t}=e^s$. Then the Donaldson's functional $M(G_t,G_{t'})$ is well defined for any $t,t' \geqslant 0$. 
We claim  that 
\begin{equation}\label{equa:Dfunctional}
 \frac{\partial}{\partial t} M(G_0,G_t) = - \int_X ||  \sqrt{-1}\Lambda_\omega F_{G_t} - \lambda_\omega(\cF) \mathrm{Id}_\cF  ||_{G_{t}}^2. 
\end{equation} 
Admitting this equality at the moment.  
With Lemma \ref{lemma:h-s-estimate} and Lemma \ref{lemma:stability-persist}, we verify that the conditions of Lemma \ref{lemma:simpson-estimate} are satisfied for $s_t$. 
Thus  the lemma shows that $||s||_{G_0}$ is bounded on $X_\cF \times [0,\infty)$. 
Hence $||g||_{G_0}$ is bounded as well. 
Thanks to Lemma \ref{lemma:Schauder-estimate-matrix-1} and local elliptic estimate, we may extract a sequence $\{t_i\}$  diverging to $\infty$ such that $g_{t_i}$ converge to a  section $g_\infty$, smoothly on any compact subset of $X_\cF$. 
It follows   that $\mathrm{det} \, g_\infty \equiv 1$ and the metric $G_\infty=G_0g_\infty$ is admissible.

By Lemma \ref{lemma:h-s-estimate}, and the properties (1) and (3) of $g$, we see that the integrals
\[
\int_{X^\circ}   ||(\dbar s_t) ||_{{G_0,\omega}}^2  \omega^n
\]
are uniformly bounded. 
From  Lemma \ref{lemma:bounded-psiM2},  we obtain that $M(G_0,G)$ is a bounded function on $[0,\infty)$. 
Therefore, the equation (\ref{equa:Dfunctional}) implies that the integrals  
\[\int_X ||  \sqrt{-1}\Lambda_\omega F_{G_t} - \lambda_\omega(\cF) \mathrm{Id}_\cF  ||_{G_{t}}^2
\] 
converge to zero when $t$ diverges to $\infty$. 
Thus $\Lambda_\omega F_{G_\infty} = \lambda_\omega(\cF) \mathrm{Id}_\cF$, and $G_\infty$ is a Hermitian-Yang-Mills metric.

It remains to prove the equality (\ref{equa:Dfunctional}) above. 
We will adapt the method of \cite[Lemma 7.1]{Simpson1988}. 
By  Lemma \ref{lemma:Dfunctional-addition}, we only need to prove the equality at $t=0$.   
We recall that
\[ M(G_0,G_0e^s) = \int_X \mathrm{tr}\, (s\sqrt{-1}\Lambda F_{G_0}) + \int_X  -\sqrt{-1}\Lambda  <\Psi(s)(\dbar s), \dbar s>_{G_0}.\]
The heat equation implies that 
\[\frac{\partial}{\partial t}\Big|_{t=0} s_t = - (\sqrt{-1}\Lambda_\omega F_{G_0} - \lambda_\omega(\cF) \mathrm{Id}_\cF). \]
Item (1) of Lemma \ref{lemma:bound-norm-h} implies that there is a constant $C$ such that \[ ||\frac{1}{t}s_t||_{G_0} \leqslant C \] for all $0<t\leqslant 1.$
Hence we get 
\begin{eqnarray*}
&&\frac{\partial}{\partial t}\Big|_{t=0} \int_X \mathrm{tr}\, (s_t\sqrt{-1}\Lambda_\omega F_{G_0}) \\
&=& \int_X \mathrm{tr}\, \Big(\frac{\partial}{\partial t}\Big|_{t=0}s_t\sqrt{-1}\Lambda_\omega F_{G_0}\Big)\\
&=& -\int_X \mathrm{tr}\, \Big(   (\sqrt{-1}\Lambda_\omega F_{G_0} - \lambda_\omega(\cF) \mathrm{Id}_\cF)    \sqrt{-1}\Lambda_\omega F_{G_0}\Big)\\
&=&  - \int_X ||  \sqrt{-1}\Lambda_\omega F_{G_0} - \lambda_\omega(\cF) \mathrm{Id}_\cF  ||_{G_{0}}^2.
\end{eqnarray*}
Thus we only need to show that 
 \[
\lim_{t\to 0} \frac{1}{t}  \int_X  -\sqrt{-1}\Lambda_\omega  <\Psi(s)(\dbar s), \dbar s>_{G_0} = 0.
 \]
We note that $s$ is bounded on $X\times [0,1]$ after property (2) of $g$. 
Thus by Lemma \ref{lemma:bounded-psiM2}, it suffices to show that 
\[
\lim_{t\to 0} \frac{1}{t}  \int_X ||\dbar s  ||_{G_0}^2 =0.
\]
By Lemma \ref{lemma:h-s-estimate}, it is enough to show that 
\[ \lim_{t\to 0} \frac{1}{t}  \int_X ||(\dbar g(t)) g(t)^{-\frac{1}{2}} ||_{G_0}^2 =0.
 \]
Thanks to property (3) of $g$, we will show that 
\[ \lim_{t\to 0} \frac{1}{t}  \int_{X}    |\mathrm{tr}\, (g  \Lambda_\omega F_{G} - g \Lambda_\omega F_{G_{0}}) |    \omega^n = 0.
 \]
From the derivative estimate of item (1) of Lemma \ref{lemma:bound-norm-h}, we get a constant $C_1$ such that, for $t\in [0,1]$, 
\[  ||g-\mathrm{Id}_\cF ||_{G_0} \leqslant C_1 t. \] 
We also remark that, since $\det\, g \equiv 1$,  the trace of  $  \Lambda F_{G} -  \Lambda F_{G_{0}} $ is zero. 
Therefore, there is a constant $C_2$, such that for $0<t\leqslant 1$,
\begin{eqnarray*}
\frac{1}{t}  \int_{X}    |\mathrm{tr}\, (g   \Lambda_\omega F_{G} - g \Lambda F_{G_{0}})|  &=& \frac{1}{t} \int_{X}    \Big|\mathrm{tr}\, \Big( (g-\mathrm{Id}_\cF) (  \Lambda_\omega F_{G} -  \Lambda_\omega F_{G_{0}}) \Big) \Big|    \omega^n \\
&\leqslant & \frac{1}{t} C_2 C_1 t \int_{X}    ||  \Lambda_\omega F_{G} - \Lambda_\omega F_{G_{0}}||_{G_0} \omega^n. 
\end{eqnarray*}
Since $\Lambda_\omega F_G$ are uniformly bounded and converge to $\Lambda_\omega F_{G_{0}}$, smoothly on any compact subset of $X_\cF$, we deduce that 
\[ \lim_{t\to 0}  \int_{X}    ||  \Lambda_\omega F_{G} - \Lambda_\omega F_{G_{0}}||_{G_0} \omega^n   =0.\]
Thus $\lim_{t\to 0} \frac{1}{t}  \int_{X}    |\mathrm{tr}\, (g  \Lambda_\omega F_{G} - g \Lambda_\omega F_{G_{0}}) |    \omega^n = 0.$
This completes the proof of the theorem.
\end{proof}

As an application, we deduce  Theorem \ref{thm:BG-inequality}.

\begin{proof}[{Proof of Theorem \ref{thm:BG-inequality}}]
By Theorem \ref{thm:existence-HE}, there exists an admissible Hermitian-Yang-Mills metric $H$ on $\cF$. Let $F$ be its Chern curvature and
\[ T = F - \frac{1}{nr} \omega (\sqrt{-1}\Lambda F). \] 
Then on $X_\cF$,  we have the following inequality (see for example \cite[Section 4.4]{Kobayashi2014}),
\[  \Big(c_2(\cF,H)-\frac{r-1}{2r}c_1(\cF,H)^2 \Big) \wedge \omega^{n-1} =  \frac{  ||T||_H^2 }{8\pi^2   n(n-1)} \omega^{n} \geqslant 0.\]  
Since $H$ is admissible, by taking the integration, we obtain the inequality. 
Furthermore, the equality holds if and only if $||T||_H \equiv 0$,  
which means that  the Chern  connection is projectively flat.
\end{proof}

\renewcommand\refname{Reference}
\bibliographystyle{alpha}
\bibliography{references}

\end{document}